\documentclass[11pt]{article}
\usepackage[margin=1in]{geometry} 
\usepackage{amssymb,amsfonts,amsmath,bbm,mathrsfs,stmaryrd}
\usepackage{xcolor}
\usepackage{url}
\usepackage{relsize}

\usepackage{enumerate}

\usepackage[colorlinks,
             linkcolor=black!75!red,
             citecolor=blue,
             pdftitle={},
             pdfauthor={x},
             pdfproducer={pdfLaTeX},
             pdfpagemode=None,
             bookmarksopen=true
             bookmarksnumbered=true]{hyperref}

\usepackage{tikz}
\usepackage{tikz-cd}
\usetikzlibrary{arrows,calc,decorations.pathreplacing,decorations.markings,intersections,shapes.geometric,through,fit,shapes.symbols,positioning,decorations.pathmorphing}

\usepackage[amsmath,thmmarks,hyperref]{ntheorem}
\usepackage{cleveref}
\usepackage{comment}

\creflabelformat{enumi}{#2(#1)#3}

\crefname{section}{Section}{Sections}
\crefformat{section}{#2Section~#1#3} 
\Crefformat{section}{#2Section~#1#3} 

\crefname{subsection}{\S}{\S\S}
\crefformat{subsection}{#2\S#1#3} 
\Crefformat{subsection}{#2\S#1#3} 

\theoremstyle{plain}

\newtheorem{lemma}{Lemma}[section]
\newtheorem{proposition}[lemma]{Proposition}
\newtheorem{corollary}[lemma]{Corollary}
\newtheorem{theorem}[lemma]{Theorem}
\newtheorem{conjecture}[lemma]{Conjecture}

\theoremstyle{nonumberplain}

\theoremstyle{plain}
\theorembodyfont{\upshape}
\theoremsymbol{\ensuremath{\blacklozenge}}

\newtheorem{definition}[lemma]{Definition}
\newtheorem{example}[lemma]{Example}
\newtheorem{remark}[lemma]{Remark}

\crefname{definition}{definition}{definitions}
\crefformat{definition}{#2definition~#1#3} 
\Crefformat{definition}{#2Definition~#1#3} 

\crefname{ex}{example}{examples}
\crefformat{example}{#2example~#1#3} 
\Crefformat{example}{#2Example~#1#3} 

\crefname{remark}{remark}{remarks}
\crefformat{remark}{#2remark~#1#3} 
\Crefformat{remark}{#2Remark~#1#3} 

\crefname{convention}{convention}{conventions}
\crefformat{convention}{#2convention~#1#3} 
\Crefformat{convention}{#2Convention~#1#3} 

\crefname{claim}{claim}{claims}
\crefformat{claim}{#2claim~#1#3} 
\Crefformat{claim}{#2Claim~#1#3} 

\crefname{conjecture}{conjecture}{conjectures}
\crefformat{conjecture}{#2conjecture~#1#3} 
\Crefformat{conjecture}{#2Conjecture~#1#3}

\crefname{lemma}{lemma}{lemmas}
\crefformat{lemma}{#2lemma~#1#3} 
\Crefformat{lemma}{#2Lemma~#1#3} 

\crefname{proposition}{proposition}{propositions}
\crefformat{proposition}{#2proposition~#1#3} 
\Crefformat{proposition}{#2Proposition~#1#3} 

\crefname{question}{question}{questions}
\crefformat{question}{#2question~#1#3} 
\Crefformat{question}{#2Question~#1#3}

\crefname{corollary}{corollary}{corollaries}
\crefformat{corollary}{#2corollary~#1#3} 
\Crefformat{corollary}{#2Corollary~#1#3} 

\crefname{theorem}{theorem}{theorems}
\crefformat{theorem}{#2theorem~#1#3} 
\Crefformat{theorem}{#2Theorem~#1#3} 

\crefname{assumption}{assumption}{Assumptions}
\crefformat{assumption}{#2assumption~#1#3} 
\Crefformat{assumption}{#2Assumption~#1#3} 

\crefname{equation}{}{}
\crefformat{equation}{(#2#1#3)} 
\Crefformat{equation}{(#2#1#3)}

\theoremstyle{nonumberplain}
\theoremsymbol{\ensuremath{\blacksquare}}

\newtheorem{proof}{Proof}

\newcommand\bC{{\mathbb C}}

\newcommand\bL{{\mathbb L}}
\newcommand\bN{{\mathbb N}}

\newcommand\bS{{\mathbb S}}

\newcommand\bZ{{\mathbb Z}}

\newcommand\cC{{\mathcal C}}

\newcommand\cK{{\mathcal K}}

\newcommand\cT{{\mathcal T}}

\def\polhk#1{\setbox0=\hbox{#1}{\ooalign{\hidewidth
    \lower1.5ex\hbox{`}\hidewidth\crcr\unhbox0}}}
    
\newcommand{\trivdim}[2]{\mathrm{dim}^{#2}_{\mathrm{LT}}(#1)}
\newcommand{\wtrivdim}[2]{\mathrm{dim}^{#2}_{\mathrm{WLT}}(#1)}
\newcommand{\strivdim}[2]{\mathrm{dim}^{#2}_{\mathrm{SLT}}(#1)}

\newcommand{\join}[3]{#1 \circledast^{#2} #3}

\newcommand{\bes}{\begin{equation*}}
\newcommand{\ees}{\end{equation*}}
\newcommand{\be}{\begin{equation}}
\newcommand{\ee}{\end{equation}}
\newcommand{\stars}[1]{\operatorname{Star}_{#1}}

\allowdisplaybreaks

\begin{document}

\title{Equivariant Dimensions of Graph C*-algebras}
\author{Alexandru Chirvasitu,
  Benjamin Passer,
  Mariusz Tobolski
}

\date{}

\newcommand{\Addresses}{{% additional braces for segregating \footnotesize
  \bigskip
  \footnotesize

  \textsc{Department of Mathematics, University at Buffalo, Buffalo,
    NY 14260-2900, USA}\par\nopagebreak \textit{E-mail address}:
  \texttt{achirvas@buffalo.edu}

  \medskip
  
   \textsc{Department of Mathematics, United States Naval Academy, Annapolis, MD 21402-5002, USA}\par\nopagebreak \textit{E-mail address}: \texttt{passer@usna.edu}

  \medskip
  \textsc{Instytut Matematyczny, Uniwersytet Wroc\l{}awski, pl. Grunwaldzki 2/4, Wroc\l{}aw, 50-384 Poland}\par\nopagebreak \textit{E-mail address}:
  \texttt{mariusz.tobolski@math.uni.wroc.pl}
}}

\maketitle

\begin{abstract}
  We explore the recently introduced local-triviality dimensions by studying gauge actions on graph $C^*$-algebras, as well as the restrictions of the gauge action to finite cyclic subgroups.  For $C^*$-algebras of finite acyclic graphs and finite cycles, we characterize the finiteness of these dimensions, and we further study the gauge actions on many examples of graph $C^*$-algebras. These include the Toeplitz algebra, Cuntz algebras, and $q$-deformed spheres.
\end{abstract}

\noindent {\em Key words:} local-triviality dimension, graph $C^*$-algebra, Cuntz algebra, Toeplitz algebra

\vspace{.5cm}

\noindent{MSC 2010: 46L55; 46L85; 19L47; 55N91}

%%%%%%%%%%%%%%%%%%%%%%%%%%%%%%%%%%%%%%%%%%%%%%%%%%%%%%%%%%%%%%%%%%%%%%%%%%%%%%%%%%%%%%
%%%%%%%%%%%%%%%%%%%%%%%%%%%%%%%%%%%%%%%%%%%%%%%%%%%%%%%%%%%%%%%%%%%%%%%%%%%%%%%%%%%%%%

\section{Introduction}

Principal bundles are fundamental objects in algebraic topology, and their local triviality is pivotal for applications in physics.  One can measure the local triviality of a given principal bundle using the Schwarz genus~\cite[Definition 5]{schwarz2}, which is the smallest number of open sets in a trivializing cover. The local-triviality dimension of \cite[Definition 3.1]{hajacindex} is then a noncommutative generalization of this invariant, which was inspired by the Rokhlin dimension~\cite[Definition~1.1]{dimrok} used in and around the classification program of unital, simple, separable, nuclear $C^*$-algebras. In this manuscript, $G$~always denotes a nontrivial compact group, though we note that many of these concepts apply to compact quantum groups as well.

\begin{definition}\label{def:trivdim}\cite[Definition 3.1]{hajacindex}
Let $G$ act on a unital $C^*$-algebra $A$, and let $t$ denote the identity function in $C_0((0, 1])$. The {\it local-triviality dimension} $\trivdim{A}{G}$ is the smallest $n$ for which it is possible to define $G$-equivariant $*$-homomorphisms $\rho_i:C_0((0,1])\otimes C(G)\to A$ for all $0 \leq i \leq n$ in such a way that $\sum_{i=0}^n \rho_i(t\otimes 1)=1$. If there is no finite $n$, the local-triviality dimension is $\infty$.
\end{definition}

Above, $G$ acts on $C(G)$ by translation: $(\alpha_g f) (h) = f(hg)$. The local-triviality dimension is a measure of an action's complexity, and finiteness of the local-triviality dimension implies freeness of the action in the sense of Ellwood~\cite[Definition 2.4]{ell}. For actions of compact abelian groups, Ellwood's freeness condition is equivalent to freeness (or saturation) in the sense of Rieffel, as in \cite[Definition~1.6]{rieffelproper} or \cite[Definition~7.1.4]{phillipsfreeness}.

The local-triviality dimension is useful for proving noncommutative generalizations of the Borsuk-Ulam theorem (see \cite[\S 6]{hajacindex}), namely, results which claim that $G$-equivariant unital $*$-homomorphisms between certain unital $C^*$-algebras do not exist. In particular, if $G$ acts on $A$ and $B$, and there is a $G$-equivariant unital $*$-homomorphism from $A$ to $B$, then $\trivdim{A}{G} \geq \trivdim{B}{G}$. These pursuits revolve around the Type 1 noncommutative Borsuk-Ulam conjecture of \cite[Conjecture 2.3]{BDH}, which is stated for coactions of compact quantum groups. While there is significant interest in the quantum case~\cite{alexbeninvariants,BDHrevisited,hajacindex}, here we will only consider actions of compact groups, so that the appropriate subcase of the conjecture is equivalent to the following.

\begin{conjecture}
Let $G$ be a nontrivial compact group which acts freely on a unital $C^*$-algebra $A$. Equip the join 
\begin{equation*}
\join{A}{}{C(G)} = \{f \in C([0, 1], A \otimes C(G)): f(0) \in \bC \otimes C(G), f(1) \in A \otimes \bC \}
\end{equation*}
with the diagonal action of $G$. Then there is no equivariant, unital $*$-homomorphism $\phi: A \to \join{A}{}{C(G)}$. 
\end{conjecture}

On the other hand, the local-triviality dimension (along with its variants discussed in \Cref{def:all_trivdim}) also generalizes an earlier approach to noncommutative Borsuk-Ulam theory. In \cite[Main Theorem]{taghavi}, Taghavi considers actions of finite abelian groups on unital Banach algebras and places restrictions on the structure of individual elements of spectral subspaces. This culminates in questions such as \cite[Question 3]{taghavi} for noncommutative spheres. From this point of view, one seeks to bound how many elements from a spectral subspace are needed to produce an invertible sum-square. In \cref{se.first_actual_section}, we recast the local-triviality dimensions in terms of such computations.

Our results focus on graph $C^*$-algebras \cite{bprs, raeburn05}. Many algebraic properties of a graph $C^*$-algebra, such as simplicity, or classification of certain ideals, can be described purely in terms of the underlying directed graph. Further, every graph $C^*$-algebra is equipped with a useful action of the circle, called the gauge action.  Freeness of this action is again determined by simple conditions on the graph, as in \cite[Proposition 2]{szym03} or \cite[Corollary 4.4]{chr}, so our concern here is to bound the local-triviality dimensions of the gauge action.

The gauge action and its restrictions to finite cyclic subgroups give many examples of locally trivial noncommutative principal bundles. By studying these bundles, we point out phenomena not found in the commutative case, such as \Cref{ex.strict}. Namely, the tensor product of non-free $\bZ/3$-actions may be free. Graph $C^*$-algebras also give a natural framework for element-based noncommutative Borsuk-Ulam theory, as the Vaksman--Soibelman quantum spheres $C(S^{2n-1}_q)$ of \cite{vs90} admit a graph presentation from \cite[Proposition 5.1]{hs-sph}. The gauge action provides an answer to Taghavi's \cite[Question 3]{taghavi}, which we examine in \Cref{pr.weakdimquantumbad}. These claims follow from a general study of the local-triviality dimension of the gauge action, which we similarly apply to other familiar graph $C^*$-algebras.

The paper is organized as follows. In \cref{se.last_blabla_section}, we recall basic facts on graph $C^*$-algebras and the local-triviality dimensions. Next, in \cref{se.first_actual_section}, we recast the local-triviality dimensions for actions of $\bZ/k$ or $\bS^1$ in terms of elements in the spectral subspaces. From this, we generate some bounds on the local-triviality dimensions of the gauge action restricted to $\bZ/2$, phrased in terms of the adjacency matrix, which we follow with a brief discussion of tensor products and $\bZ/k$ actions.  In \Cref{se.fin-acyclic-wdim}, we show that both the local-triviality dimension and its stronger version can only take the value $0$ or $\infty$ for the gauge $\mathbb{Z}/k$-action on a $C^*$-algebra of a finite acyclic graph.  For the same actions, however, finiteness of the weak local-triviality dimension is equivalent to freeness. \Cref{se.ex} contains various examples for which we can obtain additional local-triviality dimension estimates, such as the Cuntz algebras, the Toeplitz algebra, the graph of an $n$-cycle, and quantum spheres realized as graph $C^*$-algebras.

%%%%%%%%%%%%%%%%%%%%%%%%%%%%%%%%%%%%%%%%%%%%%%%%%%%%%%%%%%%%%%%%%%%%%%%%%%%%%
%%%%%%%%%%%%%%%%%%%%%%%%%%%%%%%%%%%%%%%%%%%%%%%%%%%%%%%%%%%%%%%%%%%%%%%%%%%%%
\section{Graph algebras and the local-triviality dimensions}\label{se.last_blabla_section}

We will consider directed graphs $E = (E^0, E^1, r, s)$, where $E^0$ is a countable set of vertices, $E^1$ is a countable set of edges, and $r, s: E^1 \to E^0$ are the range and source maps, respectively. The adjacency matrix $A_E$, defined by 
\begin{equation*}
(A_E)_{vw} = \# \text{ edges with source } v \text{ and range } w,
\end{equation*}
has entries in $\mathbb{Z}^+ \cup \{\infty\}$. In particular, loops and distinct edges with the same source and range are allowed.

\begin{definition}\label{def.graphalgebra} The \textit{graph} $C^*$\textit{-algebra} $C^*(E)$ is the universal $C^*$-algebra generated by elements $P_v$ for all $v \in E^0$ and elements $S_e$ for all $e \in E^1$, subject to the following constraints:

\begin{enumerate}
\item\label{cond:gph_vertex} for each $v \in E^0$, $P_v$ is a projection, and these projections are mutually orthogonal,
\item\label{cond:gph_edge} for each $e \in E^1$, $S_e$ is a partial isometry, and the ranges of these partial isometries are mutually orthogonal,
\item\label{cond:gph_range} for each $e \in E^1$, $S_e^* S_e = P_{r(e)}$,
\item\label{cond:gph_ineq} for each $e \in E^1$, $S_e S_e^* \leq P_{s(e)}$, and
\item\label{cond:gph_sum} for each $v \in E^0$ with $s^{-1}(v)$ nonempty and finite, $P_v = \sum\limits_{e \in s^{-1}(v)} S_eS_e^*$.
\end{enumerate}
\end{definition}

We will primarily consider finite graphs, where $E^0$ and $E^1$ are both finite sets, so item \ref{cond:gph_sum} will apply provided $v$ is not a sink (and hence condition \ref{cond:gph_ineq} is automatic in this case). Unsurprisingly, in many of our computations, sinks must be handled separately. We will consider paths (i.e., directed paths) from left to right: a path of length $n \geq 1$ is a composition $e_1 \cdots e_n$ of edges with $r(e_j) = s(e_{j+1})$ for each $1 \leq j \leq n - 1$, and a path of length $0$ is just a vertex $v$. We denote the length of a path $\mu$ by $|\mu|$. Some particular consequences of the above relations include
\begin{itemize}
\item distinct edges have $S_e^* S_f = 0$, regardless of whether or not the ranges or sources agree, and
\item edges with distinct ranges have $S_e S_f^* = 0$.
\end{itemize}
Further, $C^*(E)$ always admits the gauge action of the circle, defined as follows.

\begin{definition}
Let $E$ be a graph. Then the \textit{gauge action} is the action of $\mathbb{S}^1$ on $C^*(E)$ generated by the automorphisms $S_e \mapsto \lambda \, S_e$ and $P_v \mapsto P_v$ for $\lambda \in \mathbb{S}^1$.
\end{definition}

We will estimate the local-triviality dimension of the gauge action, and of the restriction of the gauge action to cyclic subgroups $\bZ/k$ of $\bS^1$. Similarly, we will consider two other variants of the local-triviality dimension. In what follows, $E_n G$ is defined by $E_0 G = G$ and $E_{n+1} G = E_n G * G$, where $*$ is the topological join operation. Also, $C(E_n G)$ is equipped with the diagonal action of $G$.

\begin{definition}\label{def:all_trivdim}\cite[Definitions 3.1 and 3.20]{hajacindex}
Let $G$ act on a unital $C^*$-algebra $A$, and let $t$ denote the identity function in $C_0((0, 1])$. 
\begin{itemize}
\item The {\it (plain) local-triviality dimension} $\trivdim{A}{G}$ is the smallest $n$ for which there exist $G$-equivariant $*$-homomorphisms $\rho_0, \ldots, \rho_n:C_0((0,1])\otimes C(G)\to A$ such that $\sum_{i=0}^n \rho_i(t\otimes 1) = 1$.
\item The {\it weak local-triviality dimension} $\wtrivdim{A}{G}$ is the smallest $n$ for which there exist $G$-equivariant $*$-homomorphisms $\rho_0, \ldots, \rho_n:C_0((0,1])\otimes C(G)\to A$ such that $\sum_{i=0}^n \rho_i(t\otimes 1)$ is invertible.
\item  The {\it strong local-triviality dimension} $\strivdim{A}{G}$ is the smallest $n$ for which there exists a $G$-equivariant, unital $*$-homomorphism from $C(E_n G)$ to $A$.
\end{itemize}
In any of the above cases, if there is no finite $n$, then the associated dimension is $\infty$.
\end{definition} 

Note that even if $\sum_{i=0}^n \rho_i(t \otimes 1)$ is invertible, it need not commute with any of its summands, so it may not be possible to normalize the sum. However, an action with $\wtrivdim{A}{G} < \infty$ is still free \cite[Theorem 3.8]{hajacindex}. Further, as in \cite[Question 3.22]{alexbeninvariants} and the preceding comments, noncommutative Borsuk-Ulam theorems have effectively only used the weak local-triviality dimension. From \cite[Proposition 3.4]{alexbeninvariants}, the plain local-triviality dimension satisfies $\trivdim{C(E_n G)}{G} \leq n$, which implies that the strong local-triviality dimension dominates the plain one. In general,

\begin{equation}\label{eq:thecomparison}
\wtrivdim{A}{G} \leq \trivdim{A}{G} \leq \strivdim{A}{G}.
\end{equation}
Since we have restricted attention to classical compact groups, we also have the following characterization.

\begin{proposition}\label{pr.sdimclassical}
Let $G$ act on a unital $C^*$-algebra $A$, and as above let $t$ denote the identity function in $C_0((0, 1])$. Then $\strivdim{A}{G}$ is the smallest $n$ for which it is possible to define $G$-equivariant $*$-homomorphisms $\rho_i: C_0((0,1]) \otimes C(G) \to A$ for all $0 \leq i \leq n$ in such a way that $\bigcup_{i=0}^n \operatorname{Ran}(\rho_i )$ generates a commutative $C^*$-algebra and $\sum_{i=0}^n \rho_i(t \otimes 1) = 1$. Equivalently, $\strivdim{A}{G}$ is the infimum of $\trivdim{C(X)}{G}$ where $C(X) \subseteq A$ is a $G$-invariant, commutative, unital $C^*$-subalgebra.
\end{proposition}
\begin{proof}
The space $E_nG$ can be described as
\begin{equation}\label{coneprod}
E_nG=\left\{([(s_0,g_0)],\ldots,[(s_n,g_n)])\in\left(\mathcal{C}G\right)^{n+1}~:~\sum_{i=0}^ns_i=1\right\}\subset \left(\mathcal{C}G\right)^{n+1},
\end{equation}
where $\cC G$ denotes the unreduced cone of $G$, so Gelfand--Naimark duality gives a surjective \mbox{$G$-equi}\-vari\-ant unital $*$-homomorphism $\pi:C(\mathcal{C}G)^{\otimes (n+1)}\to C(E_nG)$. 

If $\strivdim{A}{G}= n$, then there is a unital, $G$-equivariant $*$-homomorphism
\[
\rho:C(E_nG)\longrightarrow A,
\]
so for each $0 \leq i \leq n$, we may define a $G$-equivariant $*$-homomorphism
\[
\rho_i:C_0((0,1])\otimes C(G)\overset{\iota_i}{\to} C(\mathcal{C}G)^{\otimes (n+1)}
\overset{\pi}{\to} C(E_nG)\overset{\rho}{\rightarrow} A.
\]
Here $\iota_i$ denotes the inclusion of $C_0((0,1])\otimes C(G)$ into the $i$th tensor factor of $C(\mathcal{C}G)^{\otimes (n+1)}$.
The images of the $\rho_i$ generate a commutative $C^*$-algebra, and $\sum_{i=0}^n \rho_i(t \otimes 1) = 1$.

Conversely, assume there exist $G$-equivariant $*$-homomorphisms $\rho_i: C_0((0,1]) \otimes C(G) \to A$ for all $0 \leq i \leq n$, such that $\bigcup_{i=0}^n \operatorname{Ran}(\rho_i )$ generates a commutative $C^*$-algebra and such that the condition $\sum_{i=0}^n \rho_i(t \otimes 1) = 1$ also holds. Since the unitization $(C_0((0,1])\otimes C(G))^+$ is $G$-equivariantly isomorphic to $C(\cC G)$ (where the action is trivial on the scalar part), we can define the $G$-equivariant $*$-homomorphisms $\rho^+_i:C(\mathcal{C}G)\to A$ for all $0\leq i\leq n$. Let $\rho^+:C(\mathcal{C}G)^{\otimes (n+1)}\to A$ denote the tensor product of the maps $\rho_i^+$, which is well-defined due to the condition on the range of each map $\rho_i$. Observe that $\rho^+$ is $G$-equivariant with respect to the diagonal action on the tensor product. Finally, by \eqref{coneprod} and the fact that $\sum_{i=0}^n \rho_i(t \otimes 1) = 1$, $\rho^+$ factors through $C(E_nG)$, providing a unital $G$-equivariant $*$-homomorphism $\rho:C(E_nG)\to A$. Note that we also showed that for commutative $G$-$C^*$-algebras, the plain and strong local-triviality dimensions coincide. 

Next, consider $\inf\trivdim{C(X)}{G}$, where the infimum is taken over all $G$-invariant, commutative, unital 
$C^*$-subalgebras of $A$. 
Since all the inclusions $C(X) \hookrightarrow A$ are $G$-equivariant,
\[
\strivdim{A}{G}\leq \strivdim{C(X)}{G}=\trivdim{C(X)}{G},
\]
and hence $\strivdim{A}{G}\leq \inf\trivdim{C(X)}{G}$. For the reverse inequality, we may assume that
\begin{equation*}
  \strivdim{A}{G} = n
\end{equation*}
is finite, in which case we denote by $C(Y)$ 
the commutative $C^*$-algebra generated by $\bigcup_{i=0}^n \operatorname{Ran}(\rho_i )$.
The subalgebra $C(Y)$ is $G$-invariant by
the equivariance of the maps $\rho_i$.
Finally,
\[
\inf\trivdim{C(X)}{G}\leq \trivdim{C(Y)}{G}=\strivdim{C(Y)}{G} \leq n = \strivdim{A}{G}.
\]
\end{proof}

\begin{remark}\label{re.w0}

When $A$ is noncommutative, the various local-triviality dimensions may take different values. However, to say $A$ has local-triviality dimension zero is unambiguous, as $E_0 G = G$ and $C(G)$ is commutative.

\begin{equation}\label{eq.donotuseverymuch}
\wtrivdim{A}{G} = 0 \hspace{.1 in} \iff \hspace{.1 in} \trivdim{A}{G} = 0 \hspace{.1 in} \iff \hspace{.1 in} \strivdim{A}{G} = 0.
\end{equation}
More generally, \cite[Theorem 3.3]{WZcpoz} may be used to show that (\ref{eq.donotuseverymuch}) holds for compact quantum groups.
\end{remark}

%%%%%%%%%%%%%%%%%%%%%%%%%%%%%%%%%%%%%%%%%%%%%%%%%%%%%%%%%%%%%%%%%%%%%%%%%%%%%%%%%%%%%%
%%%%%%%%%%%%%%%%%%%%%%%%%%%%%%%%%%%%%%%%%%%%%%%%%%%%%%%%%%%%%%%%%%%%%%%%%%%%%%%%%%%%%%
\section{The local-triviality dimensions and spectral subspaces}\label{se.first_actual_section}

An action  $\alpha$ of a compact abelian group $G$ on a unital $C^*$-algebra $A$ induces a grading of $A$ by spectral subspaces $A_\lambda$, which are given by
\begin{equation*}
A_\lambda = \{ a \in A: \text{for all } g \in G, \alpha_g(a) = \lambda(g) a\}
\end{equation*}
for characters $\lambda \in \widehat{G}$. The action $\alpha$ is free if and only if $1 \in A_\lambda A_\lambda^*$ for each $\lambda$~\cite[Theorem~7.1.15]{phillipsfreeness}. A special case of this is the translation action of $G$ on $C(G)$, $(\alpha_g f) (h) = f(hg)$. If $C(G)$ is identified with $C^*(\widehat{G})$ in the natural way, then any character $\lambda \in \widehat{G}$ belongs to its own spectral subspace $C(G)_\lambda$. 

The local-triviality dimensions we need may similarly be recast in terms of saturation properties. For $\bZ = \widehat{\bS^1}$ and $\bZ/k = \widehat{\bZ/k}$, we temporarily use multiplicative notation and denote any fixed generator by $\gamma$ in order to appropriately match the identity element with the $C^*$-algebra unit.

\begin{proposition}\label{pr.normaldegcounting}
Let $\mathbb{S}^1$ act on a unital $C^*$-algebra $A$, and let $\gamma$ denote a generator of $\mathbb{Z} = \widehat{\bS^1}$. Then 
\begin{equation*}
\trivdim{A}{\bS^1} = \min\left\{n \in \mathbb{N}: \exists \, \operatorname{ normal } \, a_0, \ldots, a_n \in A_\gamma \emph{ such that } \sum\limits_{i=0}^n a_i a_i^* = 1\right\},
\end{equation*} 
\begin{equation*}
\wtrivdim{A}{\bS^1} = \min\left\{n \in \mathbb{N}: \exists \, \operatorname{ normal } \,a_0, \ldots, a_n \in A_\gamma \emph{ such that } \sum\limits_{i=0}^n a_i a_i^* \, \operatorname{ is } \, \operatorname{ invertible}\right\},
\end{equation*} 
and
\begin{equation*}
\strivdim{A}{\bS^1} = \min\left\{n \in \mathbb{N}: \exists \, \operatorname{ normal} \, \operatorname{commuting} \, a_0, \ldots, a_n \in A_\gamma \emph{ such that } \sum\limits_{i=0}^n a_i a_i^* = 1\right\}.
\end{equation*} 
\end{proposition}
\begin{proof}
We characterize images $\rho(t \otimes 1)$ of $\bS^1$-equivariant $*$-homomorphisms $\rho: C_0((0, 1]) \otimes C(\bS^1) \to A$ as exactly those elements which may be written as $aa^*$, where $a$ is a normal contraction in $A_\gamma$.

Since $\rho$ is contractive and has commutative domain, $b := \rho(t \otimes \gamma)$ is a normal contraction, and equivariance of $\rho$ shows that $b \in A_\gamma$. A simple computation with the functional calculus shows that $a := b |b|^{-1/2}$ is a normal contraction in $A_\gamma$ with $aa^* = \rho(t \otimes 1)$.

Conversely, suppose $a \in A_\gamma$ is a normal contraction. Then $\sigma(a)$ is $\bS^1$-invariant, with
\bes
\sigma(a) = \bigcup_{s \in K} s \cdot \bS^1
\ees
for some compact $K \subseteq [0, 1]$. We then have from the spectral theorem that
\bes
C^*(a) \cong \{f \in C(K, C(\bS^1)): f(0) = 0 \text{ if } 0 \in K\} \subseteq C(K) \otimes C(\bS^1)
\ees
via the isomorphism which represents $a$ as the function $a(s) = s \gamma$. On the other hand, consider the map $\rho$ defined by composing

\[ C_0((0, 1]) \otimes C(\bS^1) \hookrightarrow C([0, 1]) \otimes C(\bS^1) \xrightarrow{\phi \otimes \text{id}} C(K) \otimes C(\bS^1), \]
where $\phi(f)(s) = f(s^2)$. Then $\rho$ is an $\bS^1$-equivariant $*$-homomorphism, and its range is a subset of $C^*(a)$. Moreover, $\rho(t \otimes 1) = aa^*$.

The above argument immediately implies the given formula for the local-triviality dimension. For the weak local-triviality dimension, we have not assumed that the $a_i$ are contractions, but this may be accomplished with a rescaling. For the strong local-triviality dimension, if the normal elements $a_i$ commute, the Fuglede-Putnam-Rosenblum theorem~\cite{rosenblum} guarantees that $a_i^* a_j = a_j a_i^*$, and hence $C^*(a_0, \ldots, a_n)$ is commutative. Therefore, the third equality follows from \Cref{pr.sdimclassical}.
\end{proof}

For actions of $\mathbb{Z}/k$, we may again recast the local-triviality dimensions in terms of the spectral subspaces, but there is an additional requirement on the elements considered. For each $k\in \bZ^+$, let
\begin{equation*}
\stars{k} := \{ se^{2 \pi i m/k}: 0 \leq s \leq 1, m \in \{0, \ldots, k - 1\} \} = \{z \in \mathbb{C}: z^k \in [0, 1]\}
\end{equation*}
denote the star-convex set with center $0$ generated by the $k$th roots of unity.

\begin{proposition}\label{pr.starkcounting}
Let $\bZ/k$ act on a unital $C^*$-algebra $A$, and let $\gamma$ denote a generator of $\bZ/k = \widehat{\bZ/k}$. Then 
\begin{equation*}
\trivdim{A}{\bZ/k} = \min\left\{n \in \mathbb{N}: \exists \, \operatorname{ normal } \, a_0, \ldots, a_n \in A_\gamma, \, \sigma(a_i) \subseteq \stars{k}, \, \sum\limits_{i=0}^n a_i a_i^* = 1\right\},
\end{equation*} 
\begin{equation*}
\wtrivdim{A}{\bZ/k} = \min\left\{n \in \mathbb{N}: \exists \, \operatorname{ normal } \, a_0, \ldots, a_n \in A_\gamma, \, \sigma(a_i) \subseteq \stars{k}, \, \sum\limits_{i=0}^n a_i a_i^*  \, \operatorname{ is } \, \operatorname{ invertible}\right\},
\end{equation*} 
and
\begin{equation*}
\strivdim{A}{\bZ/k} = \min\left\{n \in \mathbb{N}: \exists \, \operatorname{ normal} \, \operatorname{commuting} a_0, \ldots, a_n \in A_\gamma, \, \sigma(a_i) \subseteq \stars{k}, \, \sum\limits_{i=0}^n a_i a_i^* = 1\right\}.
\end{equation*}
\end{proposition}
\begin{proof}
The proof is essentially identical to that of the previous proposition, with the exception that the normal element $\rho(t \otimes \gamma)$ must have spectrum in $\stars{k}$. To see this, note that $(\rho(t \otimes \gamma))^k = \rho(t^k \otimes 1)$ is a positive contraction, and apply the spectral mapping theorem.

Similarly, any normal element $a \in A_\gamma$ has $\mathbb{Z}/k$-invariant spectrum. If, in addition, $\sigma(a) \subseteq \stars{k}$, then we may decompose the spectrum as
\begin{equation*}
\sigma(a) = \bigcup_{s \in K} s \cdot \{\zeta \in \bS^1: \zeta^k = 1\}
\end{equation*}
for some $K \subseteq [0, 1]$. The rest of the proof is then analogous to the previous case.
\end{proof}

This formulation of local-triviality dimension is reminiscent of the results and questions in \cite[\S 3]{taghavi} on graded Banach and $C^*$-algebras, which had significant influence on the development of noncommutative Borsuk-Ulam theory. Note also that when $k = 2$, a normal element $a$ with $\sigma(a) \subseteq \stars{2}$ is simply a self-adjoint contraction. Consequently, estimating the local-triviality dimensions for $\mathbb{Z}/2$-actions is considerably more tractable than the other cases, as the following result on graph $C^*$-algebras illustrates. 

\begin{lemma}\label{le.rowcolsumZ2}
Suppose $E$ is a finite graph with adjacency matrix $A_E$ such that $\begin{pmatrix} 1 & 1 &  \cdots & 1 \end{pmatrix}$ is in the row space of $I + A_E$, as a nonnegative combination of nonzero rows $v \in X \subseteq E^0$ (that is, no $v \in X$ is a sink). If $n =  \max\limits_{w \in E^0} \sum_{v \in X} (A_E)_{vw}$, then $\trivdim{C^*(E)}{\bZ/2} \leq 2n - 1$. 
\end{lemma}
\begin{proof}
By assumption, there are nonnegative coefficients $\chi_v$, $v \in X$ such that for each $w \in E^0$,
\begin{equation*}
 \sum_{v \in X} \chi_v (I + A_E)_{vw} = 1.
\end{equation*}
Since no $v \in X$ is a sink, we have that $0 < |s^{-1}(v)| < \infty$, and hence
\[
\sum_{e \in s^{-1}(X)} \chi_{s(e)} ( S_e S_e^* + S_e^* S_e) \,\, = \,\, \sum\limits_{v \in X} \chi_v \, \left[ \sum_{e \in s^{-1}(v)} \, S_e S_e^* + \sum_{e \in s^{-1}(v)} \, S_e^* S_e \right] \, = \]
\[ \sum\limits_{v \in X} \chi_v \, \left[ P_v + \sum_{e \in s^{-1}(v)} \,P_{r(e)} \right] \,\, = \,\, \sum_{v \in X} \chi_v \left[  P_v  + \sum_{w \in E^0} \sum_{e \in  s^{-1}(v) \cap r^{-1}(w)} P_w  \right] \,=  \]
\[ \sum_{v \in X} \chi_v \left[  P_v + \sum_{w \in E^0} (A_E)_{vw} \, P_w \right] \,\, = \,\, \sum_{v \in X} \chi_v \sum_{w \in E^0} (I + A_E)_{vw} P_w = \]
\[ \sum_{w \in E^0} \left[ \sum_{v \in X} \chi_v (I + A_E)_{vw} \right] \, P_w \,\, = \,\,   \sum_{w \in E^0} P_w \, = \, 1.\]

Next, consider the largest sum $n =  \max\limits_{w \in E^0} \sum_{v \in X} (A_E)_{vw}$. We may partition $s^{-1}(X)$ into subsets $Y_1, \ldots, Y_n$ such that for each fixed $Y_i$, no two edges in $Y_i$ have the same range. This implies that $S_e S_f^* = 0$ for any distinct edges $e, f \in Y_i$, in addition to the automatic condition $S_e^* S_f = 0$. It follows that $B_i := \sqrt{2} \sum_{e \in Y_i} \sqrt{\chi_{s(e)}} \, S_e$, which is in the $-1$ spectral subspace, satisfies
\begin{equation*}
\sum\limits_{i=1}^n \text{Re}(B_i)^2 + \text{Im}(B_i)^2 = \sum\limits_{i=1}^n \frac{1}{2} \, (B_i B_i^* + B_i^* B_i) = \sum\limits_{i=1}^n \sum_{e \in Y_i} \chi_{s(e)} (S_e S_e^* + S_e^* S_e) = 1,
\end{equation*}
and hence $\trivdim{C^*(E)}{\bZ/2} \leq 2n - 1$ by \Cref{pr.starkcounting}.
\end{proof}

Some assumption about sinks is necessary in \Cref{le.rowcolsumZ2}. First, note that if $E$ has an isolated vertex $v$, then the gauge $\bZ/2$-action is such that the $-1$ spectral subspace generates a proper ideal of $C^*(E)$, so the action is not free (consistent with the result \cite[Corollary 4.4]{chr} on Leavitt path algebras) and hence has infinite local-triviality dimension. However, if $A_E = \begin{pmatrix} 0 & 0 \\ 0 & 1\end{pmatrix}$, then $I  + A_E = \begin{pmatrix} 1 & 0 \\ 0 & 2 \end{pmatrix}$ still has $\begin{pmatrix} 1 & 1 \end{pmatrix}$ in the positive cone of its row space. 

\begin{theorem}\label{th.z2-lt}
  Suppose $E$ is a finite graph with no sinks, such that no
  two edges share the same source and range. If $\begin{pmatrix} 1 & 1 &  \cdots & 1 \end{pmatrix}$ is a nonnegative combination of rows of $I + A_E$, then
\begin{equation*} \trivdim{C^*(E)}{\bZ/2} \leq 2 \, |E^0| - 1. \end{equation*}
\end{theorem}
\begin{proof}
Since no two edges have the same source and range, each entry of the adjacency matrix is at most one. Therefore, $n :=  \max\limits_{w \in E^0} \sum_{v \in X} (A_E)_{vw}$, as computed in \Cref{le.rowcolsumZ2}, is at most $|E^0|$.
\end{proof}

To estimate the \textit{weak} local-triviality dimension, one may replace $\begin{pmatrix} 1 & 1 &  \cdots & 1 \end{pmatrix}$ with any vector $\begin{pmatrix} a_1 & a_2 & \ldots & a_n \end{pmatrix}$ such that each $a_i$ is strictly positive, ultimately leading to the following partition argument.

\begin{proposition}\label{pr.z2-fin}
Let $E$ be a finite graph, and let $\bZ/2$ act on $C^*(E)$ via restriction of the gauge action. Suppose there are collections $Y_1, \ldots, Y_n$ of edges such that no two edges in $Y_j$ have the same range, and for each $v \in E^0$, either
\begin{itemize}
\item $r^{-1}(v) \cap \bigcup\limits_{j=1}^n Y_j \not= \varnothing$, or
\item $s^{-1}(v) \not= \varnothing$ and $s^{-1}(v) \subseteq \bigcup\limits_{j=1}^n Y_j$.
\end{itemize}
Then $\wtrivdim{C^*(E)}{\bZ/2} \leq 2n - 1$.
\end{proposition}
\begin{proof}
Define $R_j := \sum_{e \in Y_j} S_e$, which is in the $-1$ spectral subspace. Then $R_j^* R_j = \sum_{e \in Y_j} S_e^* S_e$ since $S_e^* S_f = 0$ for $e \not= f$. Similarly, since $r(e) \not= r(f)$ for any $e \not= f$ in $Y_j$, we also have $R_j R_j^* = \sum_{e \in Y_j} S_e S_e^*$. Items \ref{cond:gph_range} and \ref{cond:gph_sum} of \Cref{def.graphalgebra} show that for each vertex $v$, either $P_v \leq R_j^*R_j$ for some $j$, or $P_v \leq \sum_{j=1}^n R_j R_j^*$. It follows that 
\[ \sum\limits_{j=1}^n 2(\text{Re}(Y_j)^2 + \text{Im}(Y_j)^2) = \sum\limits_{j=1}^n R_j R_j^* + R_j^* R_j\] 
dominates each vertex projection, hence it is invertible. The result follows from \Cref{pr.starkcounting}.
\end{proof}

\Cref{pr.z2-fin} and the previous remarks imply that for finite graphs,
\begin{equation}
\wtrivdim{C^*(E)}{\bZ/2} < \infty \,\,\, \iff \,\,\, \bZ/2 \curvearrowright C^*(E) \text{ freely} \,\,\, \iff \,\,\, E \text{ has no isolated vertices},
\end{equation}
where equivalence of the second and third items also follows from \cite[Corollary 4.4]{chr}. We may then select $n = |E^1|$ in \Cref{pr.z2-fin}.

\begin{proposition}\label{pr.z2-isolatedreferee}
Let $E$ be a finite graph, and let $\bZ/2$ act on $C^*(E)$ via restriction of the gauge action. If $E$ has no isolated vertices, then $\wtrivdim{C^*(E)}{\bZ/2} \leq 2 \, |E^1| - 1$.
\end{proposition}
\begin{proof}
Apply \Cref{pr.z2-fin}, where each $Y_j$ consists of a single edge.
\end{proof}

We will also have some use for the following simple observation on the local-triviality dimension of a tensor product.

\begin{lemma}\label{le.tens-dim}
Suppose $A$ and $B$ are unital $C^*$-algebras with actions of $\bZ/k$. Then for any choice of $C^*$-algebra tensor product, the diagonal action satisfies
  \begin{equation}\label{eq:1}
    \wtrivdim{A\otimes B}{\bZ/k} \le \min\left(\wtrivdim{A}{\bZ/k},\wtrivdim{B}{\bZ/k}\right),
  \end{equation}
and similarly for the plain and strong local-triviality dimensions.
\end{lemma}
\begin{proof}
If the elements $a_0, \ldots, a_n \in A$ witness $\wtrivdim{A}{\bZ/k}\le n$ in the language of \Cref{pr.starkcounting}, then the elements $a_0 \otimes 1, \ldots, a_n \otimes 1 \in A\otimes B$ witness $\wtrivdim{A \otimes B}{\bZ/k} \leq n$. The right tensor factor follows a similar pattern.
\end{proof}

\Cref{ex.strict} shows that the inequality in \Cref{le.tens-dim} can be strict. In fact, it is possible for $\bZ/k$ to act non-freely on $A$ and $B$ and freely on their tensor product as soon as $k\ge 3$.

%%%%%%%%%%%%%%%%%%%%%%%%%%%%%%%%%%%%%%%%%%%%%%%%%%%%%%%%%%%%%%%%%%%%%%%%%%%%%%%%%%%%%%
%%%%%%%%%%%%%%%%%%%%%%%%%%%%%%%%%%%%%%%%%%%%%%%%%%%%%%%%%%%%%%%%%%%%%%%%%%%%%%%%%%%%%%
\section{Finite acyclic graphs}\label{se.fin-acyclic-wdim}

In this section, we consider any graph $E$ which is finite and acyclic. By  \cite[Proposition 4.3]{chr}, the gauge $\bZ/k$-action on $C^*(E)$ is free if and only if for every sink $v$ of $E$, there is a path of length $k - 1$ which ends at $v$. In this case, we seek to compute the local-triviality dimensions using graph properties. 

If $E$ is finite and acyclic, then $C^*(E)$ is finite-dimensional, so it decomposes as a direct sum of matrix algebras. These summands are indexed by the sinks of $E$. However, we caution the reader that the summands need not be \textit{equivariantly} isomorphic to the graph $C^*$-algebra of an $n$-vertex path, since $M_n$ may be represented as a graph $C^*$-algebra in many different ways. However, the local-triviality dimension of a direct sum is the maximum local-triviality dimension of the summands, so this procedure reduces the problem to graphs with one sink. We may also use \cite[Theorem 3.2]{OutSplit} to replace $E$ with its maximally out-split version $F$, whose graph $C^*$-algebra is gauge-equivariantly isomorphic with~$C^*(E)$. Every path in $E$ extends to a path ending at a sink, and $F$ is such that
\begin{itemize}
\item $F^0$ is the set of directed paths in $E$ terminating at a sink, and 
\item there is some $e\in F^1$ with $s(e)=\mu$ and $r(e)=\nu$ iff $|\nu|=|\mu|-1$ and we may write $\mu=\alpha\nu$. 
\end{itemize}
The connected components of $F$ are in bijection with the sinks of $E$.

Fix $k \in \mathbb{Z}^+$. For any residue $i\in \{0,\ldots,k-1\}$ modulo $k$ and any sink $v\in E^0$, let $\#^v_i$ be the number of elements of $F^0$ terminating at $v$ and having length $i \, (\mathrm{mod}\ k)$. Also, define 
\begin{equation}\label{eq:quot}
  \mathrm{quot}^v := \left\lceil \frac{\max_i \#^v_i}{\min_j \#^v_j}\right\rceil,
\end{equation}
where the value is $\infty$ if the denominator vanishes.

\begin{theorem}\label{th.fin-acyclic-wdim}
  For a finite acyclic graph $E$, we have
  \begin{equation*}
    \wtrivdim{C^*(E)}{\bZ/k} = \max_{\mathrm{sinks} \,\, v}(\mathrm{quot}^v-1). 
  \end{equation*}
\end{theorem}
\begin{proof} The algebra $A=C^*(E) \cong C^*(F)$ is a direct sum of matrix algebras indexed by the sinks of $E$. Moreover, $F$ has one connected component $F_i$ for each sink $v_i$, and the gauge $\bZ/k$-action on $C^*(F)=\bigoplus\limits_i C^*(F_i)$ is diagonal. Since the direct sum decomposition gives $\wtrivdim{C^*(F)}{\bZ/k} = \max_i \wtrivdim{C^*(F_i)}{\bZ/k}$, we will henceforth assume that $E$ has a single sink, denoted $v$.

  Let $n := |F^0|$, and let $p_1, \ldots, p_n$ be the vertex projections for $F$, where $p_i$ is associated to the vertex $v_i$. We have that $A\cong M_n$, and the $p_i$ form a complete set of mutually orthogonal minimal projections. If $\mu_i$ denotes the unique path from $v_i$ to the sink $v$ for each $1 \leq i \leq n$, then the gauge action of $\bZ/k$ is implemented via conjugation by the unitary $u\in A$ whose action on the range of $p_i$ is multiplication by $\omega^{|\mu_i|}$, where $\omega=\mathrm{exp}(2\pi i /k)$. To see this, let $e_1, \ldots, e_n$ form a basis for the vector space $\bC^n$ upon which $M_n$ acts, with the following requirement. If an edge $f$ goes from vertex $v_j$ to vertex $v_k$, then $S_f$ is defined by $e_j \mapsto e_k$. Note that in this case, $|\mu_j| = |\mu_k| + 1$, so conjugation by the unitary $u$ defined above scales $S_f$ by $\omega$. This means that conjugation by $u$ agrees with the gauge action on the edge partial isometries. Since $M_n$ is generated as a $C^*$-algebra by these partial isometries, the conclusion follows.

  The weak local-triviality dimension is the minimal $d$ for which we can find normal contractions $x_0, \ldots, x_d$ such that $\sum\limits_{j=0}^d x_jx_j^*$ is invertible and such that for each $j \in \{0, \ldots, d\}$, $\sigma(x_j) \subseteq \stars{k}$ and $ux_ju^*=\omega x_j$. The last condition means that $x_j$ maps the $\omega^i$ eigenspace of $u$ into the $\omega^{i+1}$ eigenspace.

  {\bf (I) The weak local-triviality dimension is at least $\mathrm{quot}^v-1$.} Each $x_j x^*_j$ leaves the eigenspaces of $u$ invariant, and each $x_jx_j^*$ has rank bounded above by $\min_i \#^v_i$. For any $0 \leq j \leq d$, let $R_j$ denote the range of the restriction of $x_j x_j^*$ to the largest eigenspace of $u$, noting that this eigenspace has dimension $\max_i \#^v_i$. The ranges $R_0, \ldots, R_d$ must span the largest eigenspace of $u$, so
  \begin{equation*}
    \max_i \#^v_i \le \sum_{j=0}^d\dim R_j \le (d+1)\min_i \#^v_i, 
  \end{equation*}
hence the desired inequality $d+1\ge \mathrm{quot}^v$ also holds.

{\bf (II) The weak local-triviality dimension is at most $\mathrm{quot}^v-1$.} We may assume $\mathrm{quot}^v$ is finite. Let $i \in \{0, \ldots, k - 1\}$ range over the residues modulo $k$, and define $d := \mathrm{quot}^v - 1$. Because the $\omega^i$ eigenspace of $u$ has dimension at most
\begin{equation*}
  (d+1)\cdot(\text{smallest dimension of an eigenspace for }u),
\end{equation*}
we can find subspaces $V_{i, j}$ for each $0 \leq i \leq k$ and $0 \leq j \leq d$ satisfying the following conditions:
\begin{equation*}
  V_{i,j}\subseteq (\omega^i\text{ eigenspace of }u),
\end{equation*}
each $V_{i,j}$ has dimension $\min_i \#^v_i$ (the smallest dimension of an eigenspace for $u$), and for each $0 \leq i \leq k -1$,
\begin{equation*}
  \sum_j V_{i,j} = \omega^i\text{ eigenspace of }u. 
\end{equation*}
Now, define $x_j$ to be a partial isometry that maps $V_{i,j}$ onto $V_{i+1,j}$ isometrically, with $i$ regarded modulo $k$ (so that $k=0$), and such that $x_j^k$ is the identity on each $V_{i,j}$. We leave it to the reader to check that the following conditions are met. Each $\sigma(x_j)$ is contained in $\{0\} \cup \{\omega^i\ |\ 0\le i\le k-1\}$, each $x_j$ is normal,  $\sum_{j=0}^d x_jx_j^*$ is invertible, and $ux_ju^*=\omega x_j$ (since $x_j$ maps the $\omega^i$ eigenspace of $u$ to the $\omega^{i+1}$ eigenspace). We therefore have $\wtrivdim{C^*(E)}{\bZ/k} \leq d = \mathrm{quot}^v - 1$.
\end{proof}

Now, \cite[Proposition 4.3]{chr} shows the gauge $\bZ/k$-action is free if and only if each sink receives a path of length $k - 1$. In this case, the denominator of $\Cref{eq:quot}$ is never zero, so we reach the following.

\begin{corollary}\label{cor.fin-dim-free}
  For a finite acyclic graph $E$ and a positive integer $k$, the following are equivalent.
  \begin{enumerate}
    \renewcommand{\labelenumi}{(\alph{enumi})}
  \item\label{item:finite_acyclic_free} The gauge action of $\bZ/k$ on $C^*(E)$ is free.
  \item $\wtrivdim{C^*(E)}{\bZ/k}$ is finite.   
  \item\label{item:finite_acyclic_counting} For each sink $v$, there exists a path terminating at $v$ of length $k - 1$.
  \end{enumerate}
\end{corollary}

Note that if $E$ is finite and acyclic, then the full gauge action on $E$ cannot be free, as otherwise the (free) $\bZ/k$ restrictions would show there are paths of arbitrary size in $E$, a contradiction. We can also use \Cref{cor.fin-dim-free} to show that the inequality in  \Cref{le.tens-dim} can be strict.

\begin{example}\label{ex.strict}
  Let $n$ and $k$ be positive integers with $2\le n<k$, and realize the matrix algebra $M_n$ as $C^*(E)$ for the path $E$ of length $n-1$ (i.e., a path of $n$ vertices and $n-1$ edges). According to \Cref{cor.fin-dim-free}, the gauge $\bZ/k$-action on $M_n$ has infinite weak local-triviality dimension. 

On the other hand, the diagonal $\bZ/k$-action on $M_n \otimes M_n$ can be identified with the gauge action on the graph algebra $C^*(F)$, where $F$ is the graph on $n^2$ vertices consisting of $n$ paths whose endpoints themselves constitute a path. In total, this graph has $n^2 - 1$ edges. 

\vspace{.1 in}

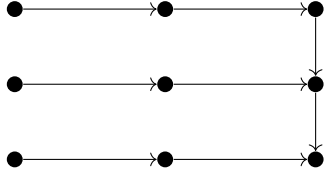
\begin{figure}[h!]
\centering
\begin{tikzpicture}
\tikzstyle{vertex}=[circle,fill=black,minimum size=6pt,inner sep=0pt]
\tikzstyle{edge}=[draw,->]
% vertices
\node[vertex] (a) at  (0,0) {};
\node[vertex] (b) at  (2,0) {};
\node[vertex] (c) at  (4,0) {};
\node[vertex] (d) at  (0, 1) {};
\node[vertex] (e) at  (2, 1) {};
\node[vertex] (f) at  (4, 1) {};
\node[vertex] (g) at  (0, 2) {};
\node[vertex] (h) at  (2, 2) {};
\node[vertex] (i) at  (4, 2) {};
%edges
\path (a) edge [edge] node {} (b);
\path (b) edge [edge] node {} (c);
\path (d) edge [edge] node {} (e);
\path (e) edge [edge] node {} (f);
\path (g) edge [edge] node {} (h);
\path (h) edge [edge] node {} (i);
\path (i) edge [edge] node {} (f);
\path (f) edge [edge] node {} (c);
\end{tikzpicture}
\caption{Graph representation of $M_3 \otimes M_3$.}
\end{figure}

There is once more a single sink, and the largest path it receives has $2n - 2$ edges. Therefore, if
  \begin{equation*}
    n<k\le 2n-1,
  \end{equation*}
  we have
  \begin{equation*}
    \wtrivdim{M_n\otimes M_n}{\bZ/k}<\infty =  \wtrivdim{M_n}{\bZ/k},
  \end{equation*}
  showing that indeed the inequality in \Cref{le.tens-dim} can be strict. 

In fact, this gives non-free $\bZ/k$-actions whose tensor products are free, even when $k\ge 3$ is prime. This is a noncommutative phenomenon, as if a prime-order group acts on spaces $X$ and $Y$ such that the diagonal action on $X\times Y$ is free, then one of the original actions must be free. Incidentally, if $\bZ/2$ acts on $C^*$-algebras $A$ and $B$, and the diagonal action on any one tensor product $A \otimes B$ is free, then the action on $A$ or $B$ is free by a direct spectral subspace argument. 
\end{example}

We now turn to the plain and strong local-triviality dimensions for finite acyclic graphs. In this case, the situation is drastically different -- there are no intermediate values between $0$ and $\infty$.

\begin{theorem}\label{th.fin-acyclic-dim}
  For a finite acyclic graph $E$, the following are equivalent. 
\begin{enumerate}
    \renewcommand{\labelenumi}{(\alph{enumi})}
\item $\strivdim{C^*(E)}{\bZ/k}$ is finite.
\item $\trivdim{C^*(E)}{\bZ/k}$ is finite.
\item For each sink $v$, $\#_i^v$ is independent of $i$.
\item $\trivdim{C^*(E)}{\bZ/k} = 0$.
\end{enumerate}
\end{theorem}
\begin{proof}
Since $\trivdim{C^*(E)}{\bZ/k} \leq \strivdim{C^*(E)}{\bZ/k}$, $(a) \Rightarrow (b)$ is trivial. Similarly, $(d) \Rightarrow (a)$ follows easily from \Cref{re.w0}. For the other two implications, we first maximally out-split $E$ to obtain a graph $F$, whose components $F_i$ are indexed by the sinks of $E$, such that
$\trivdim{C^*(F)}{\bZ/k} = \max_i \trivdim{C^*(F_i)}{\bZ/k}$. We may therefore assume $E$ is maximally out-split with one sink, $v$. 

{\bf (c) \hspace{.001 in} $\Rightarrow$ (d):} Assume that the $\#^v_i$ are equal for each $0\le i\le k-1$. By \Cref{th.fin-acyclic-wdim}, we have that $\wtrivdim{C^*(E)}{\bZ/k}$ is zero. \Cref{re.w0} then implies that all of the local-triviality dimensions are zero.

  {\bf (b) $\Rightarrow$ (c):} Suppose the plain local-triviality dimension, $d$, is finite, so there exist normal contractions $x_j$ for all $0 \leq j \leq d$ such that $\sigma(x_j) \subseteq \stars{k}$, $ux_ju^*=\omega x_j$, and $\sum\limits_{j=0}^d x_jx_j^*=1$. Let $u$ be a unitary implementing the $\bZ/k$-action, as in the proof of \Cref{th.fin-acyclic-wdim}.

  The claim we need to prove is that the eigenspaces of $u$ all have the same dimension, since for each $0 \leq i \leq k -1$,
  \begin{equation*}
    \dim\ker(u-\omega^i) = \#^v_i. 
  \end{equation*}
  As observed before, each $x_j$ maps $V_i=\ker(u-\omega^i)$ to $\ker(u-\omega^{i+1})=V_{i+1}$, and hence each $x_j x_j^*$ leaves the spaces $V_i$ invariant. Regarding $x_j$ as an operator $V_i\to V_{i+1}$, the traces of $x_j x^*_j\in \mathrm{End}(V_{i+1})$ and $x^*_j x_j\in \mathrm{End}(V_i)$
are equal. Summing over $j$ and using $\sum x_j x^*_j = \sum x^*_j x_j=1$,
we obtain
\begin{equation*}
  \dim(V_i) = \mathrm{tr}(\mathrm{id}_{V_i}) = \mathrm{tr}(\mathrm{id}_{V_{i+1}}) = \dim(V_{i+1}).
\end{equation*}
This is valid for all $i\in \bZ/k$, hence the $V_i$ are equidimensional, and $\#^v_i$ is independent of $i$.
\end{proof}

 The simplest example of a finite acyclic graph $E$ is an $n$-vertex path, and in this case $C^*(E)$ is isomorphic to the matrix algebra $M_n$.

\vspace{.1 in}

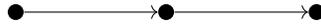
\begin{figure}[h!]
\begin{center}
\begin{tikzpicture}[auto,swap]
\tikzstyle{vertex}=[circle,fill=black,minimum size=6pt,inner sep=0pt]
\tikzstyle{edge}=[draw,->]

    \node[vertex] (1) at (0,0) {};
    \node[vertex] (2) at (2,0) {};
    \node[vertex] (3) at (4,0) {};
    \path (1) edge [edge] node {} (2);
    \path (2) edge [edge] node {} (3);
\end{tikzpicture}
\end{center}
\caption{Graph representation of $M_3$.}
\end{figure}

\noindent The gauge $\bZ/k$-action is then conjugation by 

\begin{equation}\label{eq.maddog}
  D := \mathrm{diag}(1,\omega,\ldots,\omega^{n-1}),
\end{equation}
where $\omega = \text{exp}(2 \pi i / k)$. 

\begin{proposition}\label{prop.mat}
  Let $M_n$ be realized as the graph $C^*$-algebra of an $n$-vertex path. We have
  \begin{equation*}
    \strivdim{M_n}{\bZ/k} = \trivdim{M_n}{\bZ/k}=
    \begin{cases}
      0,\ & k \, | \, n\\
      \infty,\ &k \nmid n
    \end{cases}.
\end{equation*}
\end{proposition}
\begin{proof}
The $n$-vertex path has a single sink $v$, and the $\#^v_i$ are all equal if and only if $k$ divides $n$. 
\end{proof}

\begin{remark}
\Cref{prop.mat} also shows that there are free actions by finite cyclic groups which nonetheless have infinite local-triviality dimension. The above action of $\bZ/k$ on $M_n$ is free if and only if $k \leq n$.
\end{remark}

\Cref{prop.mat} may be proved by simple dimension-counting methods, since elements of the $\omega^m$ spectral subspace of $C^*(E) = M_n$ are matrices $N$ such that $DND^{-1} = \omega^m N$. These direct methods are in fact built into the proof of the more general result \Cref{th.fin-acyclic-dim}. We will need specific reference to the local-triviality dimension of $M_n$ in subsection \ref{subse.cyc}, which concerns the graph of an $n$-cycle.

%%%%%%%%%%%%%%%%%%%%%%%%%%%%%%%%%%%%%%%%%%%%%%%%%%%%%%%%%%%%%%%%%%%%%%%%%%%%%%%%%%%%%%
%%%%%%%%%%%%%%%%%%%%%%%%%%%%%%%%%%%%%%%%%%%%%%%%%%%%%%%%%%%%%%%%%%%%%%%%%%%%%%%%%%%%%%
\section{Various examples}\label{se.ex}

%%%%%%%%%%%%%%%%%%%%%%%%%%%%%%%%%%%%%%%%%%%%%%%%%%%%%%%%%%%%%%%%%%%%%%%%%%%%%%%%%%%%%%
\subsection{Cuntz algebras}\label{subse.ctz}

The Cuntz algebra $\mathcal{O}_n$ is given by the presentation $C^*(S_1, \ldots, S_n \, | \, \sum S_i S_i^* = 1,   \,\, \forall i \,\, S_i ^* S_i = 1)$, which corresponds to a graph with a single vertex and $n$ loops.

\vspace{.15 cm}

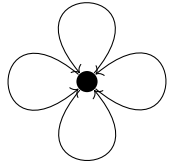
\begin{figure}[h!]
\begin{center}
\begin{tikzpicture}[auto,swap]
\tikzstyle{vertex}=[circle,fill=black,minimum size=8pt,inner sep=0pt]
\tikzstyle{edge}=[draw,->]
\tikzstyle{loop1}=[draw,->,in=130,out=50,looseness=22]
\tikzstyle{loop2}=[draw,->,in=220,out=140,looseness=22]
\tikzstyle{loop3}=[draw,->,in=310,out=230,looseness=22]
\tikzstyle{loop4}=[draw,->,in=40,out=320,looseness=22]
    \node[vertex] (1) at (0,0) {};
    \path (1) edge [loop1] node {} (1);
    \path (1) edge [loop2] node {} (1);
    \path (1) edge [loop3] node {} (1);
    \path (1) edge [loop4] node {} (1);
\end{tikzpicture}
\end{center}
\caption{Graph representation of $\mathcal{O}_4$.}
\end{figure}

\vspace{.3 cm}

We will regard the (free) gauge action $\alpha$ on $\mathcal{O}_n$ as a $\bZ$-grading, with degree $d$ component
\begin{equation*}
  \mathcal{O}_n(d):=\{x\in \mathcal{O}_n\ |\ \alpha_z(x) = z^dx,\ \forall z\in \bS^1\}. 
\end{equation*}
The $C^*$-subalgebra $F=\mathcal{O}_n(0)\subset \mathcal{O}_n$ is the UHF algebra of type $n^\infty$, the direct limit of inclusions
\begin{equation*}
  M_{n^k} \hookrightarrow M_{n^{k+1}},
\end{equation*}
where each inclusion is diagonal and unital (see e.g. \cite[$\S$1]{ctz-ctz}). As in \Cref{pr.normaldegcounting}, the local-triviality dimensions are determined by the existence of normal degree $1$ elements satisfying sum conditions. However, these elements do not exist.

\begin{proposition}\label{pr.ctz-inf}
For $n \geq 2$, $\mathcal{O}_n$ admits no normal, degree one elements other than $0$. Consequently, the gauge action of $\bS^1$ on $\mathcal{O}_n$ has infinite weak local-triviality dimension.
\end{proposition}
\begin{proof}
  Let $x\in \mathcal{O}_n(1)$ be a normal element, and let $S=S_1$ be one of the $n$ generating isometries. It follows from \cite[$\S$1.6]{ctz-ctz} that $x$ is of the form $yS$, where $y \in F$. Since right multiplication by $S$ annihilates the orthogonal complement of $SS^*$, we may assume that $y=ySS^*$. We have 
  \begin{equation}\label{eq:ys}
    yy^*=ySS^*y^*=xx^*=x^*x=S^*y^*yS,
  \end{equation}
so multiplying on the left and right by $S$ and $S^*$, respectively, and using $ySS^*=y$ shows that
  \begin{equation}\label{eq:as}
    Syy^*S^*=y^*y. 
  \end{equation}
If $\tau$ is the canonical trace on $F$, $S (\cdot) S^*$ is a (non-unital) endomorphism of $F$ that scales $\tau$ by $\frac 1n$, so
  \begin{equation*}
    \tau(y^*y) = \tau(Syy^*S^*) = \frac 1n \, \tau(yy^*) = \frac 1n \, \tau(y^*y),
  \end{equation*}
  and hence $y^*y$ is annihilated by $\tau$. Since $\tau$ is faithful, $y=0$, and hence $x=0$.
\end{proof}

We now consider the restriction of the gauge action to finite cyclic subgroups $\bZ/k$, beginning with $k = 2$.

\begin{proposition}
For $n \geq 2$, $\trivdim{\mathcal{O}_n}{\bZ/2} \leq 2n-1$.
\end{proposition}
\begin{proof}
The adjacency matrix is $[n]$, which meets the conditions of \Cref{le.rowcolsumZ2}. The formula therein reduces to examining the elements $x_i = \sqrt{\cfrac{2}{(n+1)}} \,\, \text{Re}(S_i)$ and $y_i = \sqrt{\cfrac{2}{(n+1)}} \,\, \text{Im}(S_i)$ for $0 \leq i \leq n$, which have $\sum\limits_{i=0}^n x_i^2 + y_i^2 = \cfrac{1}{n+1} \sum\limits_{i=0}^n (S_i^* S_i + S_i S_i^*) = \cfrac{1}{n + 1} \, (n + 1) = 1$. Therefore, $\trivdim{\mathcal{O}_n}{\bZ/2} \leq 2n - 1$.
\end{proof}

On the other hand, the gauge $\bZ/k$-action can never have local-triviality dimension zero.

\begin{theorem}\label{th.cuntz-ge1}
  For $n\ge 2$ and $k\ge 2$, $\wtrivdim{\mathcal{O}_n}{\bZ/k}\ge 1$.
\end{theorem}
\begin{proof}If $\wtrivdim{\mathcal{O}_n}{\bZ/k} = 0$, then there is a unitary $u$ such that $u^k = 1$ and $u$ is in the spectral subspace associated to the generator $1 \in \bZ/k$, so $\mathcal{O}_n \cong \mathcal{O}_{n}^{\bZ/k} \rtimes \bZ/k$ by \cite[Theorem 4]{landstad}. The universal property of the Cuntz algebras (see \cite[$\S$1.9]{ctz-ctz}) implies $\mathcal{O}_n^{\bZ/k}\cong \mathcal{O}_{n^{k}}$, and hence 
\begin{equation*}
\mathcal{O}_n \cong \mathcal{O}_{n^k} \rtimes \bZ/k.
\end{equation*}
We may then embed $\mathcal{O}_n$ into $M_k(\mathcal{O}_{n^k})$ by applying Takai duality \cite[Theorem 10.1.2]{blk}:
\begin{equation}\label{eq:o2o4}
 \mathcal{O}_n \hookrightarrow \mathcal{O}_n\rtimes \widehat{\bZ/k}\cong (\mathcal{O}_{n^{k}}\rtimes \bZ/k)\rtimes \widehat{\bZ/k}\cong M_{k}(\mathcal{O}_{n^{k}}).
\end{equation}
This embedding sends $1 \in \mathcal{O}_n$ to the identity matrix $I_k \in M_k(\mathcal{O}_{n^k})$, so the induced map $K_0(\mathcal{O}_{n}) \to K_0(\mathcal{O}_{n^k})$ sends $[1]$ to $k \cdot [1]$. This is a contradiction, as the order of $[1] \in K_0(\mathcal{O}_n)$ is $n - 1$, the order of $[1] \in K_0(\mathcal{O}_{n^k})$ is $n^k - 1$, and $k(n - 1) < n^k - 1$.
\end{proof}

If $k = 2$, then $x := \text{Re}(S_1)$ and $y := \text{Im}(S_1)$ have $x^2 + y^2 \geq \frac{1}{2} S_1^* S_1 = \frac{1}{2}$, so $\wtrivdim{\mathcal{O}_n}{\bZ/2}$ is precisely equal to $1$. Below, we extend this claim to $\bZ/2^m$. The next result uses the notion of a {\it Kirchberg algebra}, i.e. a $C^*$-algebra that is separable, simple, nuclear, and purely infinite (for every $x \ne 0$ and $y$ there are $a, b$ such that $a x b=y$), as described in \cite[Introduction]{enders}.

\begin{proposition}\label{pr.ctz-z2}
For $n\ge 2$ and $m\geq 1$,  $\wtrivdim{\mathcal{O}_n}{\bZ/2^m}=1$. 
\end{proposition}
\begin{proof}
The sole theorem of \cite{outercuntz} implies that the gauge action on $\mathcal{O}_n$ is pointwise outer, as are its restrictions to finite subgroups. That is, for every $g\in \bS^1\setminus\{1\}$, the automorphism $\alpha_g$ is not implemented as ${\rm Ad}(u)$ for a unitary $u\in \mathcal{O}_n$. Since Cuntz algebras are Kirchberg algebras, we use \cite[Theorem~4.19]{rokdim} to deduce that for any $k \geq 2$, the Rokhlin dimension of the gauge $\bZ/k$-action on $\mathcal{O}_n$ is at most $1$.

Next, $C_0((0, 1])$ is a separable, projective $C^*$-algebra, so \cite[Proposition~2.10]{Semiprojective} shows that the action of $\bZ/2^m$ on $C_0((0, 1]) \otimes C(\bZ/2^m)$ through translation on the right tensor factor is projective. By \cite[Proposition 5.12]{hajacindex}, we conclude that $\wtrivdim{\mathcal{O}_n}{\bZ/2^m}$ is at most the Rokhlin dimension, so $\wtrivdim{\mathcal{O}_n}{\bZ/2^m} \leq 1$. The reverse inequality follows from \Cref{th.cuntz-ge1}.
\end{proof}

%%%%%%%%%%%%%%%%%%%%%%%%%%%%%%%%%%%%%%%%%%%%%%%%%%%%%%%%%%%%%%%%%%%%%%%%%%%%%%%%%%%%%%
\subsection{The Toeplitz algebra}\label{se.tplz}

The Toeplitz algebra $\cT = C^*(V \, | \, V^*V = 1)$ is determined by the following graph.

\vspace{.15 cm}

\begin{figure}[h!]
\begin{center}
\begin{tikzpicture}[auto,swap]
\tikzstyle{vertex}=[circle,fill=black,minimum size=6pt,inner sep=0pt]
\tikzstyle{edge}=[draw,->]
\tikzset{every loop/.style={min distance=20mm,in=135,out=45,looseness= 6}}
    \node[vertex] (1) at (0,0) {};
    \node[vertex] (2) at (1.5,0) {};
    \path (1) edge [edge, loop above] node {} (1);
    \path (1) edge [edge] node {} (2);
\end{tikzpicture}
\end{center}
\caption{Graph representation of $\cT$.}
\end{figure}

\vspace{.3 cm}

\subsubsection{Strong local-triviality dimension}

In this subsection, which culminates in  \Cref{th.tplz-inf}, we show that $\strivdim{\cT}{\bZ/k} = \infty$ for all $k \geq 2$. We begin with an informal description. 

Let $\pi: \cT \to C(\bS^1)$ be the symbol map, and note that the kernel of $\pi$ is isomorphic to $\cK$, the $C^*$-algebra of compact operators.
The Toeplitz extension, namely the extension
\begin{equation}\label{eq:4}
  0\to \cK\to \cT\to C(\bS^1)\to 0
\end{equation}
of $C(\bS^1)$ by $\cK$, can be made $\bZ/k$-equivariant by realizing all three algebras as graph $C^*$-algebras and considering their gauge $\bZ/k$-actions. If $\strivdim{\cT}{\bZ/k} = n < \infty$, then by \Cref{pr.starkcounting}, there exist normal commuting elements $a_i \in \cT$ for each $0 \leq i \leq n$, such that $\sum\limits_{i=0}^n a_i a_i^* = 1$ and each $a_i$ is in the spectral subspace associated to the generator $1 \in \bZ/k$. That is, there exists a
$\bZ/k$-equivariant, unital $*$-homomorphism from $C(\bS^{2n+1})$ to $\cT$ defined by $z_i \mapsto a_{i-1}$ for each $1 \leq i \leq n+1$. Here, $z_i$ denotes the $i$th complex coordinate function on $\bS^{2n+1}$. Equivalently, the $\bZ/k$-equivariant, unital $*$-homomorphism $\Psi: C(\bS^{2n + 1}) \to C(\bS^1)$ defined by $\Psi(z_i) = \pi(a_{i-1})$ admits a $\bZ/k$-equivariant lift to $\cT$. We will use equivariant $K$-theory and $K$-homology to reach a contradiction, namely by showing that the equivariant map $\Psi$ cannot have an equivariant lift to $\cT$.

We next give a short review of equivariant $K$-theory and $K$-homology in topology. For a compact group $G$ acting on a compact space $X$, the equivariant $K$-theory groups $K_G^0(X)$ and $K_G^1(X)$ were defined using classes of vector $G$-bundles 
over $X$ by Atiyah and Segal~\cite{segal,atiyahsegal} based on the earlier work of Atiyah and Hirzebruch~\cite{atiyahhirzebruch}. 
$K$-homology is the homology theory dual to \mbox{$K$-theory}, and it has a geometric description (which uses manifolds and bordisms) due to Baum and Douglas~\cite{baumdouglas}. This description allows for an equivariant generalization~\cite{boohs,bhs}.
In the non-equivariant case, Brown, Douglas, and Fillmore (BDF)~\cite{bdf} provided an equivalent analytic description of $K$-homology of compact metric spaces using classes of $C^*$-extensions of the form 
\[
0\to\cK\to A\to C(X)\to 0,
\]
where $A$ is a unital separable $C^*$-algebra containing compact operators. In fact, this is why the first $K$-homology group of a compact space $X$ is often denoted by ${\rm Ext}(X)$ instead of $K_1(X)$. To obtain a $G$-equivariant version of BDF theory, one has to introduce a $G$-action on $\cK$. Thomsen essentially gave such an action in~\cite{thm-kk}, wherein he considered the $C^*$-algebra of compact operators on the countably infinite direct sum $\bigoplus_{i=0}^\infty L^2(G)$. In the same paper, Thomsen also showed the equivalence of equivariant BDF theory with geometric equivariant $K$-homology~\cite[Theorem~9.2]{thm-kk}\footnote{In fact, Thomsen showed that equivariant BDF theory is equivalent to the $KK$-theoretic description of equivariant $K$-homology due to Kasparov~\cite{kk2}. The equivalence with geometric $K$-homology then follows from~\cite[Theorem~3.11]{boohs} or~\cite[Theorem~4.6]{bhs}.}. 
Here, we are interested in Thomsen's construction in the case $G=\bZ/k$. Let $\widetilde{\cK}_k$ denote the set of compact operators on the countably infinite direct sum $\bigoplus_{i=0}^\infty L^2(\bZ/k)$, and let $\bZ/k$ act on each summand via the regular representation. The equivariant $K$-homology group $K^{\bZ/k}_1(X)$ is then identified with classes of equivariant $C^*$-extensions of the form
\begin{equation}\label{eq.extensionsbutwrongone}
0\to\widetilde{\cK}_k\to A\to C(X)\to 0.
\end{equation}

Since we work with graph $C^*$-algebras, we will need a seemingly different realization of the compact operators: the $C^*$-algebra $\cK$ is the graph $C^*$-algebra associated to an infinite path. The resulting gauge $\bZ/k$-action can be described as follows. First, fix a primitive $k$th root of unity $\zeta$ such that the generator $\sigma$ of $\bZ/k$ scales the edge partial isometries of the path by $\zeta$. Realize
\begin{equation*}
  \cK=\cK(\ell^2(\bN)),
\end{equation*}
and fix the standard orthonormal basis $\{e_i\}_{i\in \bN}$ for $\ell^2(\bN)$. The rank-one ``matrix units'' $e_{ij}$ are defined by
\begin{equation*}
  e_{ij} e_k = \delta_{j,k} e_i,
\end{equation*}
where
\begin{equation*}
  \delta_{j,k}=
  \begin{cases}
    1 &\text{ if }j=k\\
    0 &\text{otherwise}
  \end{cases}
\end{equation*}
is the Kronecker delta. Moreover, $\sigma$ scales $e_{ij}$ by $\zeta^{i-j}$. The action extends to all of $\cK$, as the latter is the closed span of the matrix units $e_{ij}$. Note that this is simply the extension to $\cK$ of the gauge actions on $M_n$, described in the discussion preceding \Cref{prop.mat}, upon writing
\begin{equation*}
  \cK=\varinjlim M_n
\end{equation*}
for the standard upper-left-corner embeddings $M_n \hookrightarrow M_{n+1}$.

\begin{proposition}\label{le.samek}
$\cK$ is $\bZ/k$-equivariantly isomorphic to $\widetilde{\cK}_k$.
\end{proposition}
\begin{proof}
Recall that $\widetilde{\cK}_k$ was defined as the set of compact operators on the direct sum $\bigoplus_{i=0}^\infty L^2(\bZ/k)$, where $\bZ/k$ acts on each summand via the regular representation. As a $\bZ/k$-$C^*$-algebra, $\widetilde{\cK}_k$ is the direct limit of the matrix algebra corner embeddings
  \begin{equation}\label{eq:mnk}
    M_{nk}(\bC)\cong B\left(\bigoplus_{i=0}^{n-1} L^2(\bZ/k)\right)\subset B\left(\bigoplus_{i=0}^{n} L^2(\bZ/k)\right)\cong M_{(n+1)k}(\bC),
  \end{equation}
where $\bZ/k$ acts via conjugation by $\mathrm{diag}(1,\zeta,\zeta^2,\cdots)$ for a primitive $k$th root of unity $\zeta$. The matrix algebra $M_{nk}(\bC)$ in~\Cref{eq:mnk} is $\bZ/k$-equivariantly isomorphic to the graph $C^*$-algebra of a path of $nk$ vertices. Since the $\bZ/k$-$C^*$-algebra $\cK$ is the direct limit of matrix algebras, given as the graph $C^*$-algebras of finite paths, $\widetilde{\cK}_k$ and $\cK$ are $\bZ/k$-equivariantly isomorphic.
\end{proof}

It follows from \Cref{le.samek} and the discussion preceding \Cref{eq.extensionsbutwrongone} that $K_1^{\bZ/k}(X)$ is identified with classes of equivariant extensions
\begin{equation}\label{eq.equiv_class_equiv}
  0\to \cK \to A\to C(X)\to 0,
\end{equation}
where $\cK$ is realized as the graph $C^*$-algebra of an infinite path and is equipped with the gauge $\bZ/k$-action.

Note that if a finite group $G$ acts freely on a compact space $X$, then we also have a natural identification
\begin{equation*}
  K_*^{G}(X)\cong K_*(X/G), 
\end{equation*} 
as per \cite[note (3.12)(ii)]{bch} (see also  \cite[Proposition 4.2.9]{val-bc}). The quotient of $\bS^1$ by $\bZ/k$ is $\bS^1$, so the equivariant and non-equivariant $K$-homology groups of the circle are isomorphic. In fact, the generators may be identified by a different argument.

\begin{lemma}\label{le.generator_extension}
For $k \geq 2$, $K_1^{\bZ/k}(\bS^1) \cong \bZ$ is generated by the equivalence class of the $\bZ/k$-equivariant extension $0\to \cK\to \cT\to C(\bS^1)\to 0$.
\end{lemma}
\begin{proof}
The non-equivariant Toeplitz extension $0 \to \cK \to \cT \to C(\bS^1) \to 0$ generates $K_1(\bS^1) \cong \bZ$ by index theory (see e.g.~\cite[p.19]{dgl-khom}). By the above discussion of quotient spaces, we also have $K_1^{\bZ/k}(\bS^1) \cong \bZ$. Next, restrict to non-equivariant $K$-homology using a group homomorphism $r:K_1^{\bZ/k}(\bS^1) \to  K_1(\bS^1)$ that forgets the $\bZ/k$-structure of a $C^*$-extension. The image of the $\bZ/k$-equivariant Toeplitz extension $0\to \cK\to \cT\to C(\bS^1)\to 0$ is the non-equivariant Toeplitz extension. Since the latter generates $K_1(\bS^1) \cong \bZ$, the former must also generate $K_1^{\bZ/k}(\bS^1) \cong \bZ$.
\end{proof}

Thus, if a $\bZ/k$-equivariant unital $*$-homomorphism $\Psi: C(\bS^{2n-1}) \to C(\bS^1)$ is nontrivial on $K_1^{\bZ/k}$, there is no $\bZ/k$-equivariant lift of $\Psi$ to $\cT$. We next show that $\bZ/k$-equivariant unital $*$-homomorphisms from $C(\bS^{2n-1})$ to $C(\bS^1)$ are always nontrivial using topological means.

As above, if a finite group $G$ acts freely on a compact space $X$, then \cite[note (3.12)(ii)]{bch} or \cite[Proposition 4.2.9]{val-bc} shows there is a natural identification
\begin{equation*}
  K_*^{G}(X)\cong K_*(X/G).
\end{equation*} 
This applies to the standard rotation action of $\bZ/k$ on $\bS^{2n+1}$, whose quotient is the lens space $\bL_k^{2n+1}=\bS^{2n+1}/(\bZ/k)$. That is,
\begin{equation}\label{eq.lens_equiv_noequiv}
  K_*^{\bZ/k}(\bS^{2n+1})\cong K_*(\bL_k^{2n+1}).
\end{equation}
Recall also that for compact metrizable $X$, the universal coefficient theorem~\cite[Theorem~2]{uctbrown} (see also \cite[Theorem 16.3.3]{blk}) fits the first $K$-homology group into an unnaturally-splitting natural exact sequence
  \begin{equation}\label{eq:new1}
    0\to \mathrm{Ext}^1_{\bZ}(K^0(X),\bZ)\to K_1(X)\to \mathrm{Hom}_{\bZ}(K^1(X),\bZ)\to 0.
  \end{equation}
  In our computations, every term will be finitely generated abelian. The group $\mathrm{Hom}_{\bZ}(K^1(X),\bZ)$ is free abelian, so the surjection in \Cref{eq:new1} splits, and hence $K_1(X) \cong \mathrm{Ext}^1_{\bZ}(K^0(X),\bZ) \oplus \mathrm{Hom}_{\bZ}(K^1(X),\bZ)$. It also follows from the discussion directly before \cite[Corollary~3.3]{hatcher} that $\mathrm{Ext}^1_{\bZ}(K^0(X),\bZ)$ is a torsion group. Since the finitely generated abelian group $K_1(X)$ is now written as the direct sum of a torsion group and a free abelian group, we conclude $\mathrm{Ext}^1_{\bZ}(K^0(X),\bZ) \cong \operatorname{tor}(K_1(X))$.

Let $E$ be the $\bZ/k$-module that is a direct sum of $n+1$ copies of the character given by a~primitive $k$th root of unity. Then the sphere module $S(E)$, viewed as a free $\bZ/k$-space, is $\bZ/k$-equivariantly isomorphic to $\bS^{2n+1}$, where we again consider the rotation action. Applying \cite[Corollary 2.7.6]{at-k} to $E$ gives an exact sequence
  \begin{equation}\label{eq:annoying}
    0\to K^1(\bL_k^{2n + 1})\to R(\bZ/k)\to R(\bZ/k)\to K^0(\bL_k^{2n+1})\to 0 \,.
  \end{equation}
  Here $R(\bZ/k)$ denotes the representation ring of $\bZ/k$, namely the free abelian group generated by the (equivalence classes of) irreducible complex representations. There is a ring isomorphism
  \begin{equation*}
    R(\bZ/k)\cong \bZ[\bZ/k],
  \end{equation*}
  and the multiplication in $R(\bZ/k)$ is given by the tensor product of representations.

\begin{lemma}\label{le.k1-emb}
  For $n\ge 2$, the equatorial inclusion $\bS^{2n-1} \hookrightarrow \bS^{2n+1}$ induces an embedding
  \begin{equation}\label{eq.oiwuerisdfhsd}
   \operatorname{tor}\left( K_{1}^{\bZ/k}(\bS^{2n-1}) \right) \hookrightarrow \operatorname{tor}\left(K_{1}^{\bZ/k}(\bS^{2n+1}) \right).
  \end{equation}
\end{lemma}
\begin{proof}
Recall that from \Cref{eq:new1} and the discussion immediately thereafter, we have $\operatorname{tor}( K_{1}^{\bZ/k}(\bS^{2n+1}))\cong \operatorname{Ext}_\bZ^1(K^0(\bL_k^{2n+1}),\bZ)$. To prove the lemma, it suffices to show that the equatorial embedding $f_n:\bL_k^{2n-1}\to \bL_k^{2n+1}$ induces a {\it surjection} on the torsion of the $K^0$ group (i.e. the reduced $K^0$ group, $\widetilde{K}^0$). Indeed, since by definition \[K^0(\bL_k^{2n\pm 1})\cong\bZ\oplus\widetilde{K}^0(\bL_k^{2n\pm 1}),\] $\operatorname{Ext}_\bZ^1(-,\bZ)$ commutes with the direct sum (see~\cite[Proposition~3.3.4]{weibel}), and $\operatorname{Ext}^1_\bZ(\bZ,\bZ)=0$ (see \cite[Exercise~2.5.2]{weibel}), we obtain that
\[
\operatorname{Ext}^1_\bZ(K^0(\bL_k^{2n\pm 1}),\bZ)\cong \operatorname{Ext}^1_\bZ(\widetilde{K}^0(\bL_k^{2n\pm 1}),\bZ).
\]
The map $f_n:\bL_k^{2n-1}\to \bL_k^{2n+1}$ induces a group homomorphism $\widetilde{f}_n:\widetilde{K}^0(\bL_k^{2n+1})\to \widetilde{K}^0(\bL_k^{2n- 1})$, which in turn (by definition of $\operatorname{Ext}$ as a derived functor~\cite[Definition~2.1.1]{weibel}) gives rise to a long exact sequence of the form
\begin{align*}
0\to \operatorname{Hom}_\bZ(\widetilde{K}^0(\bL_k^{2n- 1}),\bZ)\to \operatorname{Hom}_\bZ(\widetilde{K}^0(\bL_k^{2n+1}),\bZ)&\to \\ \to\operatorname{Hom}_\bZ(\ker f_n,\bZ)\to \operatorname{Ext}^1_\bZ(\widetilde{K}^0(\bL_k^{2n-1}),\bZ)&\to\operatorname{Ext}^1_\bZ(\widetilde{K}^0(\bL_k^{2n+1}),\bZ)\to\cdots\,.
\end{align*}
Since $\ker f_n$ is torsion, the group $\operatorname{Hom}_\bZ(\ker f_n,\bZ)$ is trivial, so the induced group homomorphism $\operatorname{Ext}^1_\bZ(\widetilde{K}^0(\bL_k^{2n-1}),\bZ)\to \operatorname{Ext}^1_\bZ(\widetilde{K}^0(\bL_k^{2n+1}),\bZ)$ is injective.

Let us now prove that $f_n$ induces a surjection on $\widetilde{K}^0$. First, the exact sequences
  \begin{equation}\label{pmsequence}
    0\to K^1(\bL_k^{2n\pm 1})\to R(\bZ/k)\to R(\bZ/k)\to K^0(\bL_k^{2n\pm1})\to 0
  \end{equation}
obtained by applying \Cref{eq:annoying} to both lens spaces are natural with respect to $f_n:\bL_k^{2n-1}\to \bL_k^{2n+1}$. Next, we follow~\cite[p.106]{at-k}. We have that $R(\bZ/k)\cong \bZ[\rho]/(\rho^k-1)$, and the middle map in~\eqref{eq:annoying} is simply multiplication by $(1-\rho)^{n+1}$. Comparing the sequences \eqref{pmsequence}, we obtain a~commutative diagram with exact top and bottom.
\[
\begin{tikzcd}
R(\bZ/k) \arrow[r,"(1-\rho)^{n+1}"] \arrow[d, "1-\rho"] & R(\bZ/k) \arrow[r] \arrow[d,"1"] & K^0(\bL_k^{2n+1}) \arrow[r] \arrow[d] & 0\\
R(\bZ/k) \arrow[r,"(1-\rho)^n"] & R(\bZ/k) \arrow[r] & K^0(\bL_k^{2n-1}) \arrow[r] & 0
\end{tikzcd}
\]
Thus, the map $K^0(\bL^{2n+1})\to K^0(\bL^{2n-1})$ is onto. Since $\widetilde{K}^0(X)$ is the quotient of $K^0(X)$ by the image of the map $K^0(\text{pt})\to K^0(X)$ (induced by a constant map $X\to \text{pt}$), the proof is complete.
\end{proof}

We can finally prove the desired nontriviality of equivariant maps on equivariant $K$-homology.

\begin{theorem}\label{th.triv}
 Let $\phi_n:\bS^1\to \bS^{2n+1}$ be $\bZ/k$-equivariant for some $n\geq 0$ and $k \geq 2$. Then the group homomorphism
  \begin{equation}\label{eq:new3}
    \phi_{n*}:\bZ\cong K_1^{\bZ/k}(\bS^1)\to K_1^{\bZ/k}(\bS^{2n+1})
  \end{equation}
  induced by $\phi_n$ in equivariant $K$-homology is nontrivial.
\end{theorem}
\begin{proof}
We consider three cases.

{\bf Case I.}  Let $n = 0$, so $\phi_0: \bS^1 \to \bS^1$. The induced homomorphism $\phi_{0*}$ depends only on the equivariant homotopy class of $\phi_0$. The associated map $\psi_0:\bL_k^1\cong\bS^1 \to\bL_k^1\cong\bS^1$ on the lens spaces is homotopically equivalent to a map which sends $z \in \bS^1$ to $z^t$. Because $\bS^1\to\bL^1_k$ is a~covering map, we can lift the homotopy so that $\phi_0$ is $\bZ/k$-equivariantly homotopically equivalent to the map $z\mapsto z^t$ as well. Since $\phi_0$ is $\bZ/k$-equivariant, $t$ must be congruent to $1$ mod $k$. Now, $\phi_0$ has degree congruent to $1$ modulo $k$, so $\phi_0$ induces a one-to-one endomorphism on the homology $H_1(\bS^1)\cong\bZ$, and similarly on the homology $H_1(\bL_k^1)$. This fits into an inclusion $\bZ\cong H_1(\bS^1)\to H_1(\bL^1_k)\cong \bZ$ of index $k$. Since $\bL^1_k$ is one-dimensional, its $K_1$ group is isomorphic to $H_1(\bL^1_k)$ by \cite[Lemma A.4.1]{val-bc}. Thus, the induced map on $K_1(\bL^1_k)$ is nontrivial, and the conclusion follows by \Cref{eq.lens_equiv_noequiv}.

{\bf Case II.} Let $n = 1$, so $\phi_1: \bS^1 \to \bS^3$. The integral homology of $\bL^3_k$ reads
  \begin{equation*}
    (H_0,H_1,H_2,H_3) = (\bZ,\bZ/k,0,\bZ). 
  \end{equation*}
  It follows that the Atiyah--Hirzebruch spectral sequence~(see~\cite[Section 2]{atiyahhirzebruch}) 
  \begin{equation*}
    E^2_{p,q}:=H_p(\bL^3_k, K_q(\text{pt}))\Rightarrow K_{p+q}(\bL^3_k)
  \end{equation*}
collapses on the third page $E^3$, giving an exact sequence $0\to H_1(\bL^3_k)\to K_1(\bL^3_k)\to B\to 0$. Here $B$ is the kernel of the differential $d^3:H_3(\bL^3_k)\to H_0(\bL^3_k)$ of the spectral sequence. The torsion of $K_1(\bL^3_k)$ is thus naturally identified with $H_1(\bL^3_k)$, and the conclusion follows from the fact that our map $\bL^1_k \to \bL^3_k$ is nontrivial on fundamental groups (which are abelian) and hence on $H_1$. The result again follows from \Cref{eq.lens_equiv_noequiv}.

{\bf Case III.} Let $n \geq 2$. We claim that the map $\phi_n:\bS^1\to\bS^{2n+1}$ is $\bZ/k$-equivariantly homotopically equivalent to a map obtained by a composition
\[ \bS^1 \xrightarrow{\widetilde{\phi}} \bS^3 \hookrightarrow \bS^5 \hookrightarrow \cdots \hookrightarrow \bS^{2n + 1}, \]
where $\widetilde{\phi}$ is $\bZ/k$-equivariant and the other maps are given by equatorial embeddings. Indeed, by the cellular  approximation theorem \cite[Theorem~4.8]{hatcher}, $\psi_n:\bL_k^1\to\bL_k^{2n+1}$ is homotopic to a map
\[ \bL_k^1\xrightarrow{\widetilde{\psi}}\bL_k^3\hookrightarrow \bL_k^5\hookrightarrow\cdots\hookrightarrow\bL_k^{2n+1}.\]
Here we consider the standard CW-complex structure of $\bL_k^{2n+1}$. Since $\bS^{2n+1}\to\bL_k^{2n+1}$ is a covering map for all $n\geq 0$, we can lift the above homotopy to a $\bZ/k$-equivariant one. Case II shows $\widetilde{\phi}_*$ is nontrivial on $K^{\bZ/k}_1$, with nontrivial image in $\text{tor}(K^{\bZ/k}_1(\bS^3))$. Finally, the equatorial embeddings $\bS^3 \hookrightarrow \bS^5 \hookrightarrow \ldots \hookrightarrow \bS^{2n+1}$ induce embeddings on the torsion of $K_1^{\bZ/k}$ by \Cref{le.k1-emb}, so $\phi_n$ induces a~nontrivial map on $K_1^{\bZ/k}$.
\end{proof}

At last, we reach the desired result on the strong local-triviality dimension of $\cT$.

\begin{theorem}\label{th.tplz-inf}
The Toeplitz algebra has strong local-triviality dimension $\strivdim{\cT}{\bZ/k} = \infty$ for $k \geq 2$.
\end{theorem}
\begin{proof}
Suppose $\strivdim{\cT}{\bZ/k} = n < \infty$. Then \Cref{pr.starkcounting} gives commuting, normal elements $a_0, \ldots, a_n$ in the spectral subspace associated to the generator $1 \in \bZ/k$, with $\sum_{j=0}^n a_i a_i^* = 1$. This gives a $\bZ/k$-equivariant unital \mbox{$*$-homo}\-morphism from $C(\bS^{2n+1})$ to $\cT$, hence also a $\bZ/k$-equivariant unital $*$-homomorphism \mbox{$\Psi: C(\bS^{2n+1}) \to C(\bS^1)$} which lifts through $\pi: \cT \to C(\bS^1)$. By \Cref{le.generator_extension} and the discussion thereafter, $\Psi$ induces a trivial map on $K_1^{\bZ/k}$. This contradicts \Cref{th.triv}.
\end{proof}

\subsubsection{Weak local-triviality dimension}

\begin{proposition}
The Toeplitz algebra has weak local-triviality dimension $\wtrivdim{\cT}{\bZ/2} = 1$.
\end{proposition}
\begin{proof}
\Cref{pr.z2-fin} applies to the graph representation of $\cT$, where the single set $Y_1$ contains both edges. It follows that $\wtrivdim{\cT}{\bZ/2}\leq 2 \cdot 1 - 1 = 1$. Since $\pi:\cT\to C(\bS^1)$ is $\bZ/2$-equivariant for the antipodal action on $C(\bS^1)$, and $\wtrivdim{\bS^1}{\bZ/2}=1$, we have $\wtrivdim{\cT}{\bZ/2} \geq 1$.
\end{proof}

%%%%%%%%%%%%%%%%%%%%%%%%%%%%%%%%%%%%%%%%%%%%%%%%%%%%%%%%%%%%%%%%%%%%%%%%%%%%%%%%%%%%%%
\subsection{Cycles}\label{subse.cyc}

The graph $C^*$-algebra of the length $n$ cycle $E = \cC_n$ is isomorphic to $M_n(C(\bS^1))\cong C(\bS^1)\otimes M_n$. 
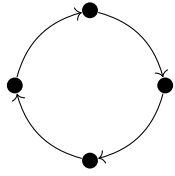
\begin{figure}[h!]
\begin{center}
\begin{tikzpicture}[auto,swap]
\tikzstyle{vertex}=[circle,fill=black,minimum size=6pt,inner sep=0pt]
\tikzstyle{edge}=[draw,->]
\tikzset{every loop/.style={min distance=20mm,in=130,out=50,looseness=50}}
    \node[vertex] (1) at (0,1) {};
    \node[vertex] (2) at (1,0) {};
    \node[vertex] (3) at (0,-1) {};
    \node[vertex] (4) at (-1,0) {};
    \path (1) edge [edge, bend left] node {} (2);
    \path (2) edge [edge, bend left] node {} (3);
    \path (3) edge [edge, bend left] node {} (4);
    \path (4) edge [edge, bend left] node {} (1);
\end{tikzpicture}
\end{center}
\caption{Graph representation of $M_4(C(\bS^1))$.}
\end{figure}

\noindent Let $v_1, \ldots, v_n$ be the vertices of $\cC_n$, with $e_i$ the edge originating at $v_i$ and $S_i$ the corresponding partial isometry. For $k\ge 2$, the gauge $\bZ/k$-action scales $S_i$ by $\omega=\exp(2\pi i/k)$ and is the tensor product of two actions:
\begin{itemize}
\item rotation by $\frac{2 n \pi i}{ k}$ (or multiplication by $\omega^n$) on $C(\bS^1)$, and 
\item the gauge $\bZ/k$-action on $M_n$, realized as the graph $C^*$-algebra of a length $n - 1$ path.  
\end{itemize}
This claim follows from the description of $C^*(\cC_n)$ given in \cite[Example 2.14]{raeburn05}. We therefore begin with a proposition concerning diagonal actions on $C(X) \otimes M_n$.

\begin{proposition}\label{pr.dimX1}
  Let $X$ be a compact connected space equipped with a free $\bZ/k$-action such that $\trivdim{C(X)}{\bZ/k}= 1$. Let $\bZ/k$ also act by the gauge action of $M_n$ from its standard graph presentation. Then we have the following possibilities for the diagonal action of $\bZ/k$ on $C(X) \otimes M_n$.
  \begin{itemize}
  \item If $k \, | \, n$, then all local-triviality dimensions are $0$.
  \item If $k \nmid n$, then all local-triviality dimensions are $1$.
  \end{itemize}
\end{proposition}
\begin{proof}
If $k \, | \, n$, then \Cref{le.tens-dim} implies $\trivdim{C(X) \otimes M_n}{\bZ/k}$ is at most $\trivdim{M_n}{\bZ/k}$, which is zero by \Cref{prop.mat}. By \Cref{re.w0}, all local-triviality dimensions of the action are zero.

Conversely, assume $\trivdim{C(X) \otimes M_n}{\bZ/k} = 0$, so there is a unital equivariant $*$-homomorphism from $C(\bZ/k)$ to $C(X) \otimes M_n \cong C(X, M_n)$. Composing this map with evaluation at an arbitrary point $x \in X$ gives a unital $*$-homomorphism from $C(\bZ/k)$ to $M_n$, which itself gives projections $p_1(x), \ldots, p_k(x) \in M_n$ with sum $I_n$. Since $X$ is connected and the action of $\bZ/k$ induces a cyclic permutation of the projections, it follows that the projections $p_1(x), \ldots, p_k(x)$ have the same trace, hence the same rank. Therefore, $k \, | \, n$.
\end{proof}

\begin{corollary}
  For the gauge action of $\bZ/k$ on $A=C^*(\cC_n)$, we have the following possibilities.
  \begin{itemize}
  \item If $k\, | \, n$, then all local-triviality dimensions are $0$.
  \item If $k \nmid n$, then all local-triviality dimensions are $1$.
  \end{itemize}
\end{corollary}
\begin{proof}
There are no nontrivial projections in $C(\bS^1)$, so for any action of a nontrivial finite group $\Gamma$ on $C(\bS^1)$, $\trivdim{C(\bS^1)}{\Gamma} > 0$. On the other hand, there is an equivariant unital $*$-homomorphism $C(E_1 \, \bZ / k) \to C(\bS^1)$ dual to a map $\bS^1 \to E_1 \, \bZ/k$, so $\trivdim{C(\bS^1)}{\bZ/k} \leq 1$. Therefore, the assumption $\trivdim{C(\bS^1)}{\bZ/k} = 1$ of \Cref{pr.dimX1} holds, and that result implies the stated conclusion.
\end{proof}

If $k$ is prime, then $X$ need not be assumed connected in \Cref{pr.dimX1}.

\begin{proposition}\label{pr.dimX1-prime}
Suppose $k$ is prime, and $X$ is a compact space equipped with a free $\bZ/k$-action such that $\trivdim{C(X)}{\bZ/k}= 1$. Let $\bZ/k$ also act by the gauge action of $M_n$ from its standard graph presentation. Then we have the following possibilities for the diagonal action of $\bZ/k$ on $C(X) \otimes M_n$.
  \begin{itemize}
  \item If $k \, | \, n$, then all local-triviality dimensions are $0$.
  \item If $k \nmid n$, then all local-triviality dimensions are $1$.
  \end{itemize}
\end{proposition}
\begin{proof}
The case when $k \, | \, n$ proceeds exactly as in \Cref{pr.dimX1}, so suppose $k \nmid n$. Set $A=C(X) \otimes M_n$, which consists of continuous maps $X\ni z\mapsto T_z\in M_n$. If $\omega = \exp(2 \pi i / k)$, then to say $T$ is in the $\omega$ spectral subspace means that $\omega T_{\omega^{-1} z} = UT_z U^{-1}$, where $U = \operatorname{diag}(\omega^i)_{i=1}^n$.

Suppose that $A$ has local-triviality dimension zero, so there is a unitary $T\in A$ with $\omega T_{\omega^{-1} z} = UT_z U^{-1}$ and $\sigma(T) \subseteq \{1, \omega, \ldots, \omega^{k-1}\}$. Let $\psi$ denote the function which sends each $z \in X$ to the spectrum of $T_z$, an $n$-\textit{multiset} of points in $\bS^1$. Equip the collection of $n$-multisets of $\bS^1$ with the action of $\bZ/k$ which simultaneously scales each element. It follows that $\psi$ is $\bZ/k$-equivariant.

  Because $k \nmid n$, $\psi(z)$ cannot contain all $k$th roots of unity with equal multiplicities. It follows from the primality of $k$ that $\bZ/k$ acts freely on the range of $\psi$, so we have a continuous $\bZ/k$-equivariant map from $X$ to a discrete free $\bZ/k$-space (consisting of unordered $n$-tuples of order $k$ roots of unity). This implies that $\trivdim{C(X)}{\bZ/k} = 0$, contradicting the hypothesis.
\end{proof}

If $k$ is composite and $X$ is not connected, one may not reach the same conclusions as in the previous propositions. Consider $k = 4$, $n = 2$, and $X = \bS^1 \sqcup \bS^1$. Let $\bZ/4$ act on $X$ in such a way that the generator $\sigma$ swaps the two circles and also implements a $\pi/4$ rotation. Since $\sigma^2$ generates the antipodal action, we may compute that $\trivdim{C(X)}{\bZ/4} = 1$. Even though $4$ is certainly not a divisor of $2$, one may show $\trivdim{C(X) \otimes M_2}{\bZ/4} = 0$ by constructing an appropriate element in the spectral subspace associated to the generator $\sigma$.

%%%%%%%%%%%%%%%%%%%%%%%%%%%%%%%%%%%%%%%%%%%%%%%%%%%%%%%%%%%%%%%%%%%%%%%%%%%%%%%%%%%%%%
\subsection{Antipodal actions on quantum spheres}

Consider the {\em Vaksman--Soibelman quantum spheres} $C(S^{2n-1}_q)$. For any $q \in (0, 1)$ and $n \in \mathbb{Z}^+$, $C(S^{2n-1}_q)$ may be written as the universal $C^*$-algebra generated by $z_1, \ldots, z_n$ subject to the following relations, as in \cite[Proposition 4.1]{vs90}. In what follows, $i$ and $j$ belong to $\{1, \ldots, n\}$.

\vspace{.15 in}

\begin{itemize}
\item If $i < j$, then $z_j z_i = q \, z_i z_j$.
\item If $i \not= j$, then  $z_j^* z_i = q \, z_i z_j^*$.
\item For all $i$, $z_i^* z_i = z_i z_i^* + (1 - q^2)\sum\limits_{j > i} z_j z_j^*$.
\item $z_1 z_1^* + \ldots + z_n z_n^* = 1$.
\end{itemize}

\vspace{.15 in}

\noindent In \cite[Theorem 4.4]{hs-sph}, Hong and Szyma\'nski showed that each $C(S^{2n-1}_q)$ is a graph $C^*$-algebra, with a graph independent of $q$. Further, the even sphere $C(S^{2n}_q)$ (the quotient of $C(S^{2n+1}_q)$ that requires $z_{n+1} = z_{n+1}^*$) also admits a $q$-independent graph presentation from \cite[Proposition 5.1]{hs-sph}.

Each quantum sphere $C(S^k_q)$ admits an \textit{antipodal} $\bZ/2$-\textit{action} given on the generators by $z_j\mapsto -z_j$ for each $1 \leq j \leq n$. Yamashita used equivariant $KK$-theory to prove the noncommutative Borsuk-Ulam theorem \cite[Corollary 15]{qdeformed}, which states that if $k < l$, then there is no $\bZ/2$-equivariant unital $*$-homomorphism from $C(S^k_q)$ to $C(S^l_q)$. In this subsection, we bound the local-triviality dimensions of quantum spheres using their presentations as graph $C^*$-algebras. Note that the classical Borsuk-Ulam theorem is equivalent to the claim that $\trivdim{C(S^k)}{\bZ/2} = k$, and computation of the (weak) local-triviality dimension for noncommutative spheres was posed as a problem, in different language, as early as \cite[Question 3]{taghavi}.

If $n$ is fixed, then $C(S^{2n-1}_q)$ is isomorphic to the graph $C^*$-algebra of a graph having $n$ vertices $v_1, \ldots, v_n$, and edges $e_{ij}$ from $v_i$ to $v_j$ for all $1 \leq i\leq j \leq n$. The antipodal action of $C(S^{2n-1}_q)$ is the gauge $\bZ/2$-action for this graph, regardless of $q$.

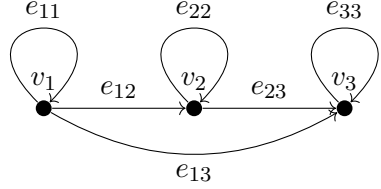
\begin{figure}[h!]
\centering
\begin{tikzpicture}
\tikzstyle{vertex}=[circle,fill=black,minimum size=6pt,inner sep=0pt]
\tikzstyle{edge}=[draw,->]
\tikzstyle{loop}=[draw,->,min distance=20mm,in=130,out=50,looseness=50]
% vertices
\node[vertex,label=above:$v_1$] (a) at  (0,0) {};
\node[vertex,label=above:$v_2$] (b) at (2, 0) {};
\node[vertex,label=above:$v_3$] (c) at (4, 0) {};
%edges
\path (a) edge [edge, in = 45, out = 135, looseness = 30] node[above] {$e_{11}$} (a);
\path (a) edge[edge] node[above] {$e_{12}$} (b);
\path (b) edge [edge, in = 45, out = 135, looseness = 30] node[above] {$e_{22}$} (b);
\path (b) edge[edge] node[above] {$e_{23}$} (c);
\path (c) edge [edge, in = 45, out = 135, looseness = 30] node[above] {$e_{33}$} (c);
\path (a) edge[edge,bend right] node[below] {$e_{13}$} (c);
\end{tikzpicture}
\caption{Graph representation of $C(S^5_q)$.}
\end{figure}

\begin{proposition}\label{oddsphere}
The antipodal $\bZ/2$-action on $C(S^{2n-1}_q)$ has $\trivdim{C(S^{2n-1}_q)}{\bZ/2} \leq 2n-1$.
\end{proposition}
\begin{proof}
The adjacency matrix $A_E$ is upper triangular, with $1$ on and above the diagonal, and $E$ has no sinks. Since $\begin{pmatrix} 1 & 1 &  \cdots & 1 \end{pmatrix} = \sum\limits_{i=1}^n 2^{-i} \,\, \text{Row}_i(I + A_E)$, the result follows from \Cref{th.z2-lt}.
\end{proof}

Next, consider the even sphere $C(S^{2n}_q)$. The graph presentation from \cite[Proposition 5.1]{hs-sph} has vertices $v_1,v_2,\ldots,v_n$ and $w_1,w_2$, edges $e_{ij}$ from $v_i$ to $v_j$ for all $1 \leq i\leq j \leq n$, and edges $f_{ik}$ from $v_i$ to $w_k$ for all $1 \leq i \leq n$ and $k \in \{1, 2\}$. For brevity, we will write $P^i:=P_{v_i}$, $P_k:=P_{w_k}$, $S^{ij}:=S_{e_{ij}}$, and $S^i_k:=S_{f_{ik}}\,$.  In this case, the antipodal action on $C(S^{2n}_q)$ does \textit{not} correspond to the gauge action, but to the following map:
\[
P^i\mapsto P^i,\qquad P_k\mapsto P_{3-k}\,,\qquad S^{ij}\mapsto -S^{ij},\qquad S^i_k\mapsto -S^i_{3-k}\,,
\]
where $i \in \{1,2,\ldots, n\}$, $j \in \{1,2,\ldots,n\}$, and $k \in \{1,2\}$. Although we are no longer considering the gauge action, we may apply a similar strategy as for \Cref{le.rowcolsumZ2} to obtain a bound for the plain local-triviality dimension.

\begin{figure}[h!]
\centering
\begin{tikzpicture}
\tikzstyle{vertex}=[circle,fill=black,minimum size=6pt,inner sep=0pt]
\tikzstyle{edge}=[draw,->]
\tikzset{every loop/.style={min distance=20mm,in=130,out=50,looseness=50}}
    \node[vertex,label=above:$v_1$] (1) at (-2,0) {};
    \node[vertex,label=above:$v_2$] (2) at (0,0) {};
    \node[vertex,label=above:$v_3$] (3) at (2,0) {};
    \node[vertex,label=below:$w_1$] (4) at (-2,-2) {};
    \node[vertex,label=below:$w_2$] (5) at (2,-2) {};
    \path (1) edge [edge, in = 45, out = 135, looseness = 30] node[above] {$e_{11}$} (1);
    \path (2) edge [edge, in = 45, out = 135, looseness = 30] node[above] {$e_{22}$} (2);
    \path (3) edge [edge, in = 45, out = 135, looseness = 30] node[above] {$e_{33}$} (3);
    \path (1) edge [edge] node[above] {$e_{12}$} (2);
    \path (1) edge [edge,bend right,looseness = 1.2] node[above] {$e_{13}$} (3);
    \path (1) edge [edge] node[left] {$f_{11}$} (4);
    \path (1) edge [edge] node[below right=0.5] {$f_{12}$} (5);
    \path (2) edge [edge] node[above] {$e_{23}$} (3);
    \path (2) edge [edge] node[left] {$f_{21}$} (4);
    \path (2) edge [edge] node[right] {$f_{22}$} (5);
    \path (3) edge [edge] node[below left=0.5] {$f_{31}$} (4);
    \path (3) edge [edge] node[right] {$f_{32}$} (5);
\end{tikzpicture}
\caption{Graph representation of $C(S^6_q)$.}
\end{figure}
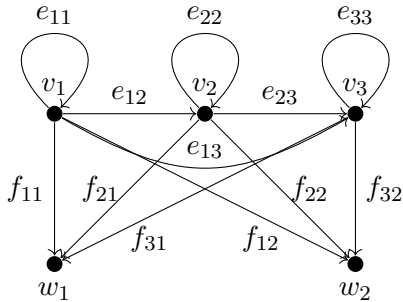

\begin{proposition}\label{evensphere}
The antipodal $\bZ/2$-action on $C(S^{2n}_q)$ has $\dim_{\rm LT}^{\bZ/2}(C(S^{2n}_q))\leq 2n$.
\end{proposition}
\begin{proof}
Define the elements $z:=\frac{1}{\sqrt{2^n}}(P_1-P_2)$ and $t_j:=\sum_{k=j}^nS_{jk}+\sum_{l=1}^2S^j_l$ for each $1 \leq j \leq n$, and consider the following rescaled real and imaginary parts of the $t_j$:
\begin{gather*}
y_{2j-1}=\frac{1}{\sqrt{2^{j+1}}}\left(\sum_{k=j}^n\left(S^*_{jk}+S_{jk}\right)+\sum_{l=1}^2(\left(S^j_l\right)^*+S^j_l)\right),\\ 
y_{2j}=\frac{1}{i\sqrt{2^{j+1}}}\left(\sum_{k=j}^n\left(S^*_{jk}-S_{jk}\right)+\sum_{l=1}^2(\left(S^j_l\right)^*-S^j_l)\right).
\end{gather*}
The elements $z$ and $y_k$, where $k$ ranges between $1$ and $2n$, are self-adjoint and in the $-1$ spectral subspace of the action. By an analogous calculation as in the proof of \Cref{le.rowcolsumZ2}, these elements satisfy $z^2 + \sum\limits_{k=1}^{2n} y_k^2 = 1$, and hence $\trivdim{C(S^{2n}_q)}{\bZ/2} \leq 2n$.
\end{proof}

It is possible that the inequalities given above are equalities, which would produce an alternative proof of \cite[Corollary 15]{qdeformed} using the local-triviality dimension. We conclude the paper with an explanation of why the \textit{weak} local-triviality dimension is insufficient for this purpose, as well as how this claim relates to earlier work on noncommutative Borsuk-Ulam theory.

In \cite[Question 3]{taghavi}, Taghavi posed the following question for noncommutative spheres. If $A$ is a noncommutative $n$-sphere, and $a_1, \ldots, a_n$ are self-adjoint elements in the $-1$ spectral subspace, must $a_1^2 + \ldots + a_n^2$ be non-invertible? Note that when $A = C(\bS^n)$, the desired conclusion is equivalent to the Borsuk-Ulam theorem. Applying \Cref{pr.starkcounting} shows that \cite[Question 3]{taghavi} has a positive answer if and only if $\wtrivdim{A}{\bZ/2} \geq n$ whenever $A$ is a noncommutative $n$-sphere. There exist multiple families of noncommutative spheres, and \cite[Theorem 2.10]{bentheta} shows that Taghavi's question has a negative answer for the $\theta$-deformed spheres. The following result shows the answer is also negative for the $q$-deformed spheres. That is, $\wtrivdim{C(S^k_q)}{\bZ/2}$ can be strictly less than $k$.

\begin{proposition}\label{pr.weakdimquantumbad}
The antipodal $\bZ/2$-action on $C(S^k_q)$ has $\wtrivdim{C(S^k_q)}{\bZ/2} = 1$.
\end{proposition}
\begin{proof}
  Any $C(S^k_q)$ admits $C(\bS^1)$ as a $\bZ/2$-equivariant quotient via the map which annihilates $z_i$ for each $i \geq 2$. Since $\wtrivdim{C(\bS^1)}{\bZ/2} = 1$, it follows that $\wtrivdim{C(S^k_q)}{\bZ/2} \geq 1$. For the reverse inequality, it suffices to consider only the odd spheres, as the even spheres are quotients.

Consider $C(S^{2n-1}_q)$ with the graph presentation above, so that the antipodal action is the gauge $\bZ/2$-action. The loops at each vertex certainly have distinct ranges, and the single set $Y_1$ containing these loops satisfies the conditions of \Cref{pr.z2-fin}, since $r^{-1}(v_i) \cap Y_1 \not= \varnothing$ for each $v_i$. It follows that $\wtrivdim{C(S^{2n-1}_q)}{\bZ/2} \leq 2 \cdot 1 - 1 = 1$.
\end{proof}

Note that the $\theta$-deformed spheres are also poorly behaved with respect to \textit{plain} local-triviality dimension, in that it is possible for $\trivdim{C(\bS^k_\theta)}{\bZ/2}$ to equal $1$ even if $k$ is large. This occurs, for example, if the generators pairwise anticommute, as in \cite[Proposition 3.21]{alexbeninvariants}. Nevertheless, the $\theta$-deformed spheres also satisfy a Borsuk-Ulam theorem (see \cite[Corollary 3.2]{bentheta} and \cite[Theorem 1.8]{benanticommuting}), which can be shown using $K$-theory.

%%%%%%%%%%%%%%

\section*{Acknowledgments}
This work is part of the project Quantum Dynamics supported by
EU-grant RISE 691246 and Polish Government grant 317281.
A.C. was partially supported by NSF grants DMS-1801011 and DMS-2001128.
M.T. was partially supported by the project Diamentowy Grant No. DI2015 006945 financed by the
Polish Ministry of Science and Higher Education.  We are grateful to the referee for helpful comments.

%%%%%%%%%%%%%%%%%%%%%%%%%%%%%%%%%%%%%%%%%%%%%%%%%%%%%%%%%%%%%%%%%%%%%%%%%%%%%%%%%%%%%%
%%%%%%%%%%%%%%%%%%%%%%%%%%%%%%%%%%%%%%%%%%%%%%%%%%%%%%%%%%%%%%%%%%%%%%%%%%%%%%%%%%%%%%

\bibliographystyle{plain}
\bibliography{dimrefs_nomi}

\Addresses

\end{document}